\documentclass[11pt]{article}
\usepackage{layout}
\usepackage{amssymb}
\usepackage{epsfig}
\usepackage{graphicx}
\usepackage{color}
\usepackage{latexsym}

\usepackage{amssymb}
\usepackage{algorithmicx}
\usepackage[ruled]{algorithm}
\usepackage{algpseudocode}
\usepackage{algc}

\alglanguage{pseudocode}
\usepackage{graphicx}
\usepackage{amsmath,epsfig,graphics}

\newenvironment{proof}[1][Proof]{\textbf{#1.} }{\ \rule{0.5em}{0.5em}}

\newtheorem{theorem}{Theorem}[section]

\newtheorem{proposition}[theorem]{Proposition}
\newtheorem{remark}[theorem]{Remark}

\newcommand{\Frac}[2] {\frac{\textstyle #1} {\textstyle #2}}

\newcommand{\ds}{\displaystyle}
\newcommand{\R}{\mathbb{R}}

\begin{document}

\title{A Stabilized bi-grid method for Allen Cahn equation in Finite Elements}
\maketitle \centerline{\scshape Hyam Abboud$^{1}$, Clara Al
Kosseifi$^{2,3}$ and Jean-Paul Chehab$^2$}
\medskip
{
  \centerline{$^{1}${\footnotesize
D\'epartement de math\'ematiques, Facult\'e des Sciences II,
Universit\'e Libanaise, Fanar, Liban}}
\centerline{$^{2}${\footnotesize Laboratoire Ami\'enois de
Math\'ematiques Fondamentales et Appliqu\'ees (LAMFA), {\small UMR
CNRS} 7352}}
  \centerline{{\footnotesize Universit\'e de Picardie Jules Verne, 33 rue Saint Leu, 80039 Amiens France} }
\centerline{$^{3}${\footnotesize Laboratoire de Physique
Appliqu\'ee (LPA), Facult\'e des Sciences II, Universit\'e
Libanaise, Fanar, Liban}}

\date{}

\begin{abstract}
In this work, we propose a bi-grid scheme framework for the
Allen-Cahn equation in Finite Element Method. The new methods are
based on the use of two FEM spaces, a coarse one and a fine one,
and on a decomposition of the solution into mean and fluctuant
parts. This separation of the scales, in both space and frequency,
allows to build a stabilization on the high modes components: the
main computational effort is concentrated on the coarse space on
which an implicit scheme is used while the fluctuant components of
the fine space are updated with a simple semi-implicit scheme;
they are  smoothed without damaging the consistency. The
numerical examples we give show the good stability and the
robustness of the new methods. An important  reduction of the
computation time is also obtained when comparing our methods with
fully implicit ones.
\end{abstract}
{\small
{\bf Keywords:} {Allen-Cahn equation, bi-grid method, stabilization, separation of the scales}\\
\hskip 0.2in{\bf  AMS Classification}[2010]: {35N57, \ 65L07,\ 65M60, \ 65N55}}\\
\section{Introduction}
Phase fields equations, such as Allen-Cahn's, are widely used in
several domains of applied sciences for modeling natural phenomena
in material science \cite{AllenCahn,Emmerich}, in image processing
\cite{LiJeongChoiLeeKim} or in ecology and in medicine
\cite{JiangShi}, just to cite but a few. They are written as
\begin{eqnarray}
\Frac{\partial u}{\partial t} +M(-\Delta u +\Frac{1}{\epsilon^2}f(u))=0, & x\in \Omega, t>0,\\
\Frac{\partial u}{\partial \nu}=0,\\
u(0,x)=u_0(x),
\end{eqnarray}
where $\Omega$ is a bounded open set in $\R^n$  and $\nu$ the unit normal vector. The Allen-Cahn
equation was introduced to describe the process of phase
separation in iron alloys \cite{AllenCahn1,AllenCahn2}, including
order-disorder transitions: $M$ is the mobility (taken to be 1 for
simplicity), $F=\displaystyle{\int_{-\infty}^u f(v)dv}$ is the
free energy, $u$ is the (non-conserved) order parameter and
$\epsilon >0$ is the interface length. The homogenous Neumann
boundary condition implies that there is no loss of mass outside
the domain $\Omega$. It is important to note that there is a
competition between the  potential term and the diffusion term:
this produces a regularization in phase transition. This equation
can be also viewed as a gradient flow for the energy
$E(u)=\ds\Frac{1}{2}\int_{\Omega}\|\nabla u\|^2dx
+\ds\Frac{1}{\epsilon^2}\int_{\Omega} F(u)dx$. A generic and
important consequence is that $\ds\Frac{d E(u)}{d t}\le 0$, the
energy is time-decreasing along the solutions, this is a stability
property that is important to be recovered numerically; we refer
the reader to recent works on numerical method for gradient flows
applied to these equations \cite{BMerletMPierre}.\\

The presence of the small parameter $\epsilon$
leads to numerical difficulties: it makes the functional $E(u)$
"very non convex" in the sense that it possesses many local minima.
Therefore, the use of semi-implicit methods suffers from a hard
time step limitation while the use of implicit schemes allows to
obtain energy diminishing methods but have an important cost in
CPU time; however,  as underlined in this paper, they present practical
difficulties for their implementation.
To combine the fast iterations of the semi-implicit schemes (only
a linear problem has to be solved at each time step) and the
good stability of the implicit schemes, stabilization methods have
been considered, \cite{JShenACCH}. They are based on a $L^2$-like
damping, but they can slow down the dynamics. This is
attribuable to the fact that the damping acts on all the components, including 
on the low mode components of the solution  which contains the main part of the information; the high mode
components represent a fluctuant part that can play a crucial role for
the numerical stability: indeed, the stability of a scheme
lies on its capability to control their expansion.\\

Independently, bi-grid methods have been extensively studied for
the solution of reaction-diffusion equation and also Navier-Stokes
\cite{AbbGirSay2,HeLiuSon}; they are based on the computation of
an approximation of the solution on the coarse space $W_H$ by
using an implicit scheme, the solution in the fine space $V_h$ is
then obtained by applying a simplified yet semi-implicit scheme. A
reduction of the CPU time is achieved, since the main
computational effort is concentred only on a small system.

The aim of this article is to propose a framework of bi-grid
methods in finite elements for the numerical approximation
of the Allen-Cahn equations.
The use of two grids, say of two FEM spaces $W_H$ and
$V_h$ with $\mbox{dim}(W_H)<< \mbox{dim}(V_h)$, allows to build a
scale separation (in space and in frequency) by decomposing the
solution $u_h \in V_h$ as
$$
u_h={\tilde u}_h+z_h.
$$
Here ${\tilde u}_h={\cal P}(u_H)$ is the $L^2$ prolongation in
$V_h$ of $u_H$ (the approximation $u_H$ of the solution in $W_H$)
is defined by
$$
(u_H-{\cal P}(u_H),\phi_h)=0, \forall \phi_h \in V_h,
$$
where $(.,.)$ denotes the scalar product in $L^2(\Omega)$. Hence
${\cal P}(u_H)\in V_h$  represents the mean part of the solution;
$z_h \in V_h$ is the fluctuant part and corresponds then to a
small correction which carries the high frequencies. In that way,
we can conjugate the bi-grid approach to reduce the CPU time  and
stabilize the correction step on the fine space by smoothing the
fluctuant component. We consider the case $W_H\subset V_h$, in a such case the method is hierarchical-like but the framework is still valid when $W_H\not \subset V_h$.\\

The article is organized as follows: at first, in Section 2, we
present briefly the phase fields problem (particularly the
Allen-Cahn equation) and we recall the classical time marching
schemes and their properties. After that, in Section 3, we
introduce the bi-grid framework as well as the separation of the
scales technique giving numerical illustrations. We then define
the bi-grid scheme with the reference scheme on the coarse space
and the correction scheme on the fine one. 
In section 4, we
present numerical results on the Allen-Cahn equation emphasizing on
the reduction of CPU time obtained by the new methods while the
solutions fit with the ones computed with the classical schemes;
we obtain particularly numerical evidences of the energy
diminishing property for the new schemes. Finally, in Section 5,
concluding remarks are given. All the numerical results were
obtained using the free Finite Element software FreeFem++
\cite{FreeFem}.
\section{Allen-Cahn equations, classical schemes. Advantages and limitations}
The Allen-Cahn equation writes as
\begin{eqnarray}
\label{AC1}
\Frac{\partial u}{\partial t} - \Delta u + \varepsilon^{-2}f(u)&=0,\  x\in \Omega , t\in (0,T),\\
\Frac{\partial u}{\partial \nu}=0 &\mbox{ on } \partial \Omega, t>0,\\
u(x_0)=u_0(x) & \mbox{ in } \Omega.
\end{eqnarray}
This equation can describe the separation of two phases, e.g; in a
metal alloy; $\epsilon>0$ represents the width of the interface
and $f=F'$ is the derivative of the potential $F$. This system can
be viewed also as a gradient flow for the energy
$E(u)=\ds\frac{1}{2}\ds\int_\Omega|\nabla
u|^2dx+\varepsilon^{-2}\ds\int_\Omega F(u)dx$ and can be rewritten
as
$$
\Frac{\partial u}{\partial t} +\nabla E(u)=0,
$$
in such a way that $\ds\frac{d E(u)}{dt}\le 0$, which guarantees
the stability of the system. This last property is very important
and has to be satisfied at the discrete level to have (energy)
stable time marching schemes; the maximum principle satisfied by
the solution is another stability property (in $L^{\infty}$ norm).
We refer to \cite{Bartels} for more details on the phase field modeling and on the mathematical properties of the solutions.\\

We now consider the time semi-discretization and focus on marching
schemes. Let $u^k\simeq u(x,k
\Delta t)$ be a  sequence of functions; $\Delta t$ is the time step. We recall the
following three classical schemes which will be used for
building our new methods. Note that a scheme is energy diminishing
when
$$
E(u^{k+1})\le E(u^k).
$$
\begin{itemize}
\item {\bf Scheme 1: Semi Implicit scheme}\\
\begin{eqnarray}
\label{AC_semi_implicit1}
\Frac{u^{k+1}-u^{k}}{\Delta t} -\Delta u^{k+1} +\Frac{1}{\varepsilon^2}f(u^k)&=0,& x\in \Omega,\\
\label{AC_semi_implicit2} \Frac{\partial u^{k+1}}{\partial \nu}=0,
&\mbox{ on } \partial \Omega.
\end{eqnarray}

This scheme is energy diminishing: it is easy to implement (only a
linear Neumann problem has to be solved at each step) but it
suffers from a hard restrictive time step condition
$$
0<\Delta t< \Frac{2 \varepsilon^2}{L},
$$
where $L=\|f'\|_{\infty}$, see e.g., \cite{JShenACCH}.
\item {\bf Scheme 2: Implicit Scheme}\\
\begin{eqnarray}
\label{AC_implicit1}
\Frac{u^{k+1}-u^{k}}{\Delta t} -\Delta u^{k+1} +\Frac{1}{\varepsilon^2}DF(u^k,u^{k+1})&=0,& x\in \Omega, \\
\label{AC_implicit2} \Frac{\partial u^{k+1}}{\partial \nu}=0,
&\mbox{ on } \partial \Omega,
\end{eqnarray}
where we have set $DF(u,v)=\left\{\begin{array}{ll}\Frac{F(u)-F(v)}{u-v} & \mbox{ if } u\neq v,\\
f(u) & \mbox{ if } u=v\end{array}\right.$.\\

This scheme is unconditionally energy stable, i.e. energy
diminishing for all $\Delta t >0$, see \cite{Elliott}. However it
necessitates to solve a fixed point problem at each step.
\item {\bf Scheme 3: Stabilized semi-implicit scheme}\\
A way to overcome the time step restriction while solving only one linear problem at each step is to
add a stabilization term as follows \cite{JShenACCH}:
\begin{eqnarray}
\label{AC_stabilized1}
\Frac{u^{k+1}-u^{k}}{\Delta t}+ \Frac{S}{\varepsilon^2}(u^{k+1}-u^k) -\Delta u^{k+1} +\Frac{1}{\varepsilon^2}f(u^k)&=0,& x\in \Omega,\\
\label{AC_stabilized2} \Frac{\partial u^{k+1}}{\partial \nu}=0,
&\mbox{ on } \partial \Omega.
\end{eqnarray}
The scheme is energy diminishing for all $\Delta t>0$ when $S\ge
\Frac{L}{2}$, see  \cite{JShenACCH}. As the semi-implicit scheme,
this method is easy to implement, however the stabilization slows
down the dynamics. One can explain it as follows: the damping term
$\Frac{S}{\varepsilon^2}(u^{k+1}-u^k)$ acts on all the modal
components of the solution, the low ones that represents the main
part of the solution and the high modes whose the limitation of
the propagation allows to obtain the stability.
\end{itemize}
An effective numerical solution of the problem needs to have a
fully discretized system, we consider here Finite Elements Method (FEM) for the
space discretization. To develop an efficient scheme we have to
take into account practical arguments, such as the implementation
as well as the cost in CPU time.
\begin{remark}
The above list of time schemes is not exhaustive, let us cite the
convex splitting scheme \cite{Eyre} which consist on decomposing
the potential as a difference of a convex and an expansive term as
$F(u)=F_c(u)-F_e(u)$. The scheme is then
$$
\Frac{u^{k+1}-u^{k}}{\Delta t} -\Delta u^{k+1} +\Frac{1}{\varepsilon^2}(\nabla F_c(u^{k+1})-\nabla F_e(u^{k}))=0.
$$
\end{remark}
\begin{remark}
An other important property of the solutions of Allen-Cahn's
equation is the maximum principle. For instance, when
$f(u)=u(1-u^2)$, it can be proved that if $|u_0(x)|\le 1,\forall x
\in \Omega $ then $|u(x,t)|\le 1,\forall x \in \Omega, \forall t
>0$: this is the $L^{\infty}$-stability.
\end{remark}
\section{Bi-grid framework}
\subsection{A scale separation in finite elements}
As stated in the introduction, the high frequency components
govern the stability of the scheme and the central idea of a
stabilized scheme is to stop or to slow down their expansion. Of
course, to this aim, we need to have a way to extract the high
mode part of the solution to stabilize, say writing $u$ as
\begin{eqnarray}\label{udecomp}
u={\tilde u}+z,
\end{eqnarray}
where ${\tilde u}$ is associated to the low mode part of $u$ and $z$,
of small size; $z$ contains the high frequencies of $u$. This
decomposition can be obtained by using several embedded
approximation spaces, as in the hierarchical methods and nonlinear
Galerkin methods
\cite{Bank,DuboisJauberteauTemam1,DuboisJauberteauTemam2,MarionXu}
and the references therein; however the embedding is not mandatory as shown hereafter. Let $V_h$ be the fine finite elements
space and let $W_H$ be another FEM space with $\mbox{dim}(W_H) <
\mbox{dim}(V_h)$. We define the prolongation operator ${\cal P}:
W_H\rightarrow V_h$ by
\begin{eqnarray}\label{uprolong}
(u_H-{\cal P} (u_H), \phi_h)=0, \forall \phi_h \in V_h.
\end{eqnarray}
Note that is not necessary to have $W_H\subset V_h$ which means that we
can avoid the building of a hierarchical basis and then the method can
be considered for many FEM spaces. However, it is important to
have compatibility conditions on $W_H$ and $V_h$ in such a way
the prolongation is uniquely defined. Let $(\phi_i)_{i=1}^N$
and $(\psi_j)_{j=1}^M$ be two bases of $V_h$ and of $W_H$
respectively ($M<N$). We define the matrix $B_H^h$ as
$(B_H^h)_{i,j}=(\phi_i,\psi_j), \ i=1,\cdots, N, \ \ j=1,\cdots,
M$. The prolongation step writes as
$$
M_h{\cal P} (u_H)=B_H^hu_H.
$$
Hence, since $M_h$ is a mass matrix,  ${\cal P}(u_H)$ is uniquely defined whenever the rank of $B_H^h$ is maximal, say equal to $M$. Of course this condition is automatically satisfied when $W_H\subset V_h$.\\
We give hereafter a sufficient condition for the injectivity of
$B_H^h$:
\begin{proposition}\label{graph_inf_sup}
Let $W_H$ and $V_h$ be two FEM spaces built on ${\cal C}^0$
reference elements. Assume that
$\forall u_H\in W_H, \left(  (u_H,\phi_h)=0,  \forall \phi_h \in
V_h \Rightarrow u_H=0\right)$
Then, $B_H^h$ is injective. Moreover, there exists two constants $\beta$ and  $\alpha_H^h>0$ such that  $0< \alpha_H^h < \beta \le 1$ and
$$
 \alpha_H^h \|u_H\| \le \|{\cal P}(u_H)\| \le \beta \| u_H\|,\forall u_H \in W_H.
$$
\end{proposition}
\begin{proof}
The assumption implies that if ${\cal P}(u_H)=0$, then
$(u_H,\phi_h)=0, \ \forall \phi_h \in V_h$, then $u_H=0$ which
gives the injectivity.

Now assuming that ${\cal P}(u_H)\neq 0$ and taking $\phi_h={\cal
P}(u_H)$ in (\ref{uprolong}), we find
$$
\|{\cal P}(u_H)\|^2=({\cal P}(u_H),u_H)=\Frac{({\cal P}(u_H),u_H)}{\|u_H\|\|{\cal P}(u_H)\|}\|u_H\|\|{\cal P}(u_H)\|.
$$
Let $\alpha_H^h= \inf_{u_H\in W_H}\Frac{({\cal
P}(u_H),u_H)}{\|u_H\|\|{\cal P}(u_H)\|}$. We now show that
$\alpha_H^h>0$. We can write
$$
\alpha_H^h= \inf_{u_H\in W_H, \|u_H\|=1}({\cal P}(u_H),u_H)\ge 0.
$$
The function $u_H\in W_H \mapsto ({\cal P}(u_H),u_H) \in\R^+$ is obviously continuous on the compact set ${\cal B}=\{u_H\in W_H, \|u_H\|=1\}$. Its minimum is reached and can not be equal to $0$ because $0$, the only root, is outside ${\cal B}$. Hence $\alpha_H^h>0$ and we infer from above that
$$
\|{\cal P}(u_H)\| \ge \alpha_H^h \|u_H\| \ge 0.
$$
Now, in the same way, we have
$$
\|{\cal P}(u_H)\|^2=({\cal P}(u_H),u_H)=\Frac{({\cal P}(u_H),u_H)}{\|u_H\|\|{\cal P}(u_H)\|}\|u_H\|\|{\cal P}(u_H)\|.
$$
But, as a consequence of the Cauchy-Schwarz inequality,
$$
\Frac{({\cal P}(u_H),u_H)}{\|u_H\|\|{\cal P}(u_H)\|}\le 1,
$$
so $\beta= \sup_{u_H\in W_H, \|u_H\|=1}({\cal P}(u_H),u_H)\le 1$.
Finally
$$
\alpha_H^h \|u_H\| \le \|{\cal P}(u_H)\| \le \beta \|u_H\|, \forall u_H \in W_H.
$$
\end{proof}
\begin{remark}
These conditions mean that the range of the angles between the
elements of $W_H$ and those of ${\cal P}(W_H)\subset V_h$ is in
the interval $[\arccos(\beta), \arccos{\alpha_H^h}]$; the
condition $\alpha_H^h >0$ avoids situations of orthogonality. Of
course $\alpha_H^h$ can depend on $H$ and $h$ and can become
smaller and smaller as $H$ and $h$ go to $0$. The best situation
being when $\alpha_H^h$ is independent of both $h$ and $H$.
\end{remark}
The correction $z$ is defined on the whole fine space $V_h$ and
not on a complementary space. Of course, one expect $z$ to be a
correction (i.e. small in norm) for regular functions $u$. More
precisely, we have the following result:
\begin{proposition}\label{z_estimates}
Let  $W_H$ and  $V_h$ be two FEM spaces that we assume to be of
class ${\cal C}^0$ and associated to nested regular triangulations
of $\Omega$, a regular bounded open set of $\R^n$; $(K,P,\Sigma)$
is the reference element. For  $u\in H^{k+1}(\Omega)$, we denote
by $u_h=\Pi_h u$ and $u_H=\Pi_H u$ the P-interpolate of $u$ in
$V_h$ and $V_H$ respectively. We assume that $P_{k}\subset P$. We
have the following estimate:
$$
\|u_h-{\cal P}u_H\|_{L^2(\Omega)} \le C H^{k+1} \|u \|_{H^{k+1}(\Omega)}.
$$
\end{proposition}
\begin{proof}
We start from the classical interpolation error estimates in FE
spaces \cite{ErnGuermond},
$$
\|u-\Pi_hu\|_{L^2(\Omega)}\le C h^{k+1} \|u \|_{H^{k+1}(\Omega)}
 \mbox{ and } \|u-\Pi_Hu\|_{L^2(\Omega)}\le C H^{k+1}\|u \|_{H^{k+1}(\Omega)}.
$$
We can write
$$
u_h-{\cal P}(u_H)=u_h-u+u-u_H+u_H-{\cal P}(u_H),
$$
hence
$$
(u_h-u+u-u_H+u_H-{\cal P}(u_H),\phi_h)=(u_h-u+u-u_H,\phi_h), \forall \phi_h \in V_h.
$$
Taking $\phi_h=u_h-{\cal P}(u_H) \in V_h$, we find,
$$
\|u_h-{\cal P}(u_H)\|^2_{L^2(\Omega)}\le \|u_h-u+u-u_H\|_{L^2(\Omega)}.\|u_h-{\cal P}(u_H)\|_{L^2(\Omega)},
$$
then
$$
\|u_h-{\cal P}(u_H)\|_{L^2(\Omega)}\le  \|u_h-u+u-u_H\|_{L^2(\Omega)}\le  \|u_h-u\|_{L^2(\Omega)}+ \|u_H-u\|_{L^2(\Omega)}.
$$
Finally
$$
\|u_h-{\cal P}(u_H)\|_{L^2(\Omega)}\le C' H^{k+1} \|u
\|_{H^{k+1}(\Omega)},
$$
where $C'>0$ is independent of $h, H$ and $u$.
\end{proof}
\begin{remark}
For $0\le m\le k+1$, defining ${\cal P}$ as $({\cal P}(u_H)-u_H, \phi_h)_m=0, \forall \phi_h \in V_h$, where $(.,.)_m$ is the standard scalar product in $H^m$, then, proceeding as above, we can prove that
$$
\|u_h-{\cal P}(u_H)\|_{H^{m}(\Omega)} \le C H^{k+1-m} \|u \|_{H^{k+1}(\Omega)}.
$$
\end{remark}
\subsection{Illustration}
For given functions and finite elements spaces $P_1$ and $P_2$, we
build the decomposition (\ref{udecomp})-(\ref{uprolong}). We
generate the meshes with FreeFem++ \cite{FreeFem} using Delaunay's
triangulation starting from the boundary. We have taken $Nf=100$
boundary points for generating ${\cal T}_h$, the mesh associated
to the fine space $V_h$ and $Nc=50$ boundary points for generating
${\cal T}_H$, the mesh associated to the coarse space $W_H$. The
first example is the decomposition of the function
$u(x,y)=(x^2+y^2-1)(x^2+y^2-4)$ on the torus; the second one is
with $u(x,y)=\sin(72x(1-x)y(1-y))$ on the unit square. Below, in
Figure \ref{exemple1zh} and  Figure \ref{exemple2zh}, we observe
that the $z$ component are much smaller in magnitude than those of
the original function ($10\%$ for $P_1$ and $0.5\%$ for $P_2$
elements).
\begin{figure}[!htp]
\begin{center}
\includegraphics[width=7.5cm, height=7.5cm]{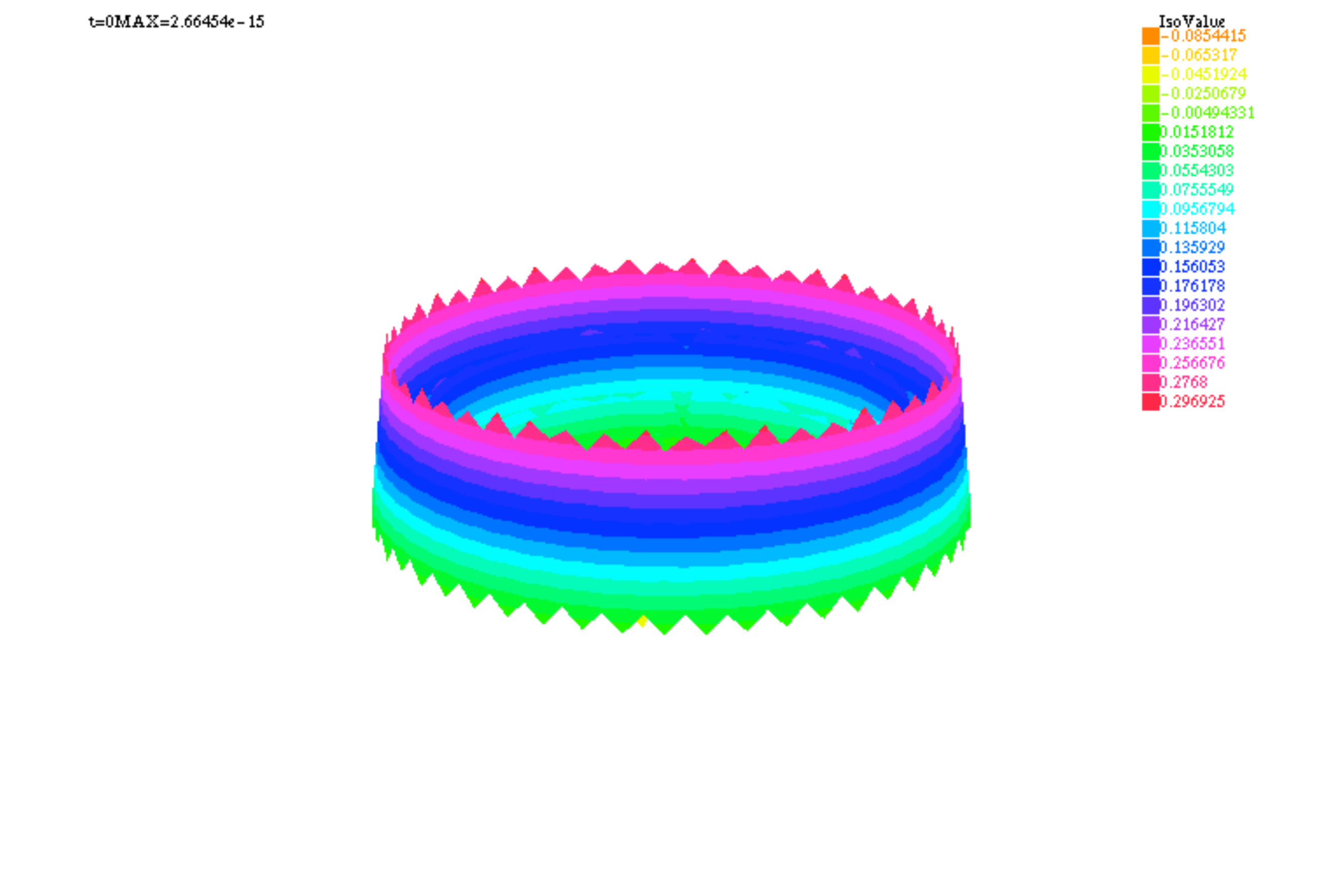}
\includegraphics[width=7.5cm, height=7.5cm]{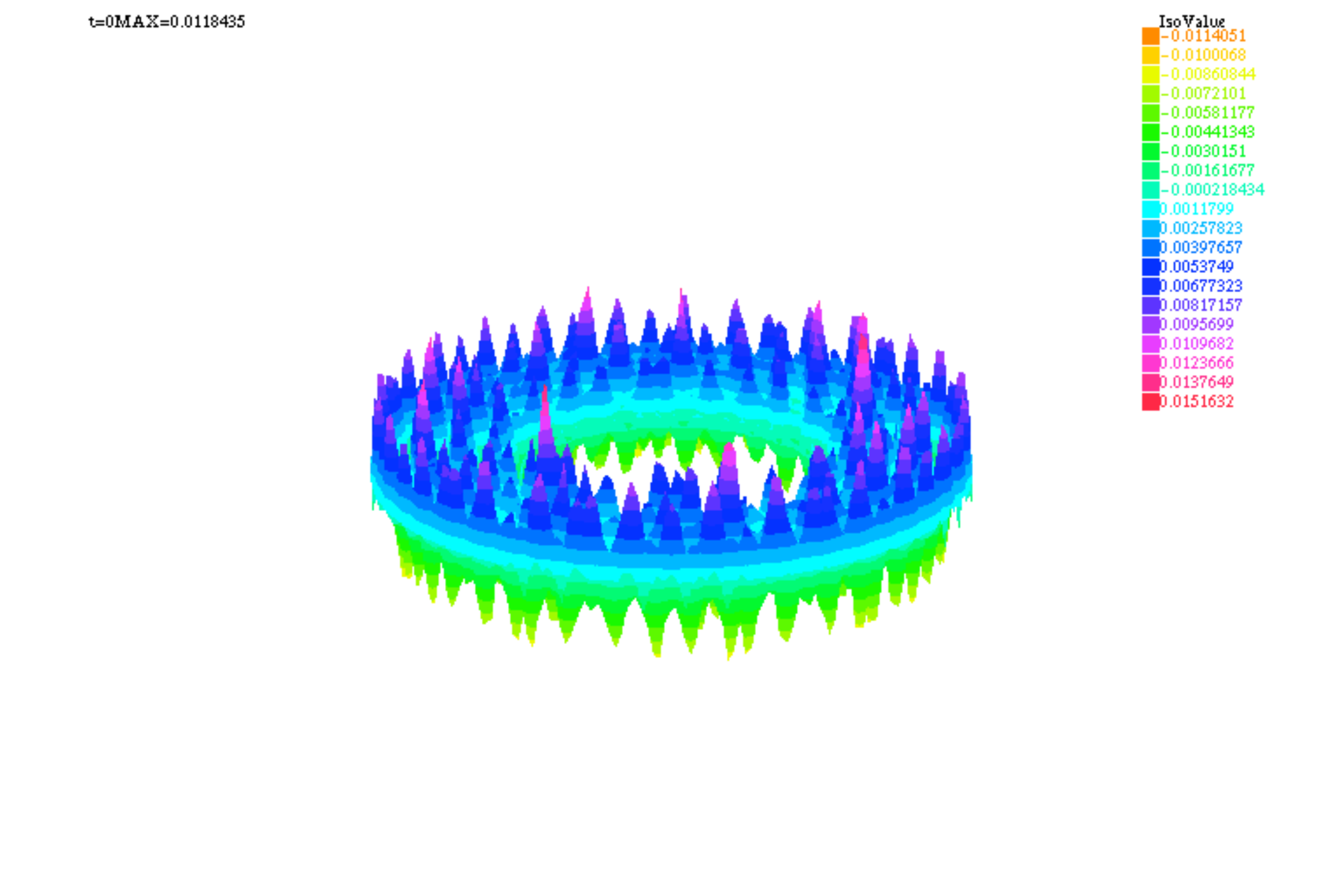}\\
\includegraphics[width=7.5cm, height=7.5cm]{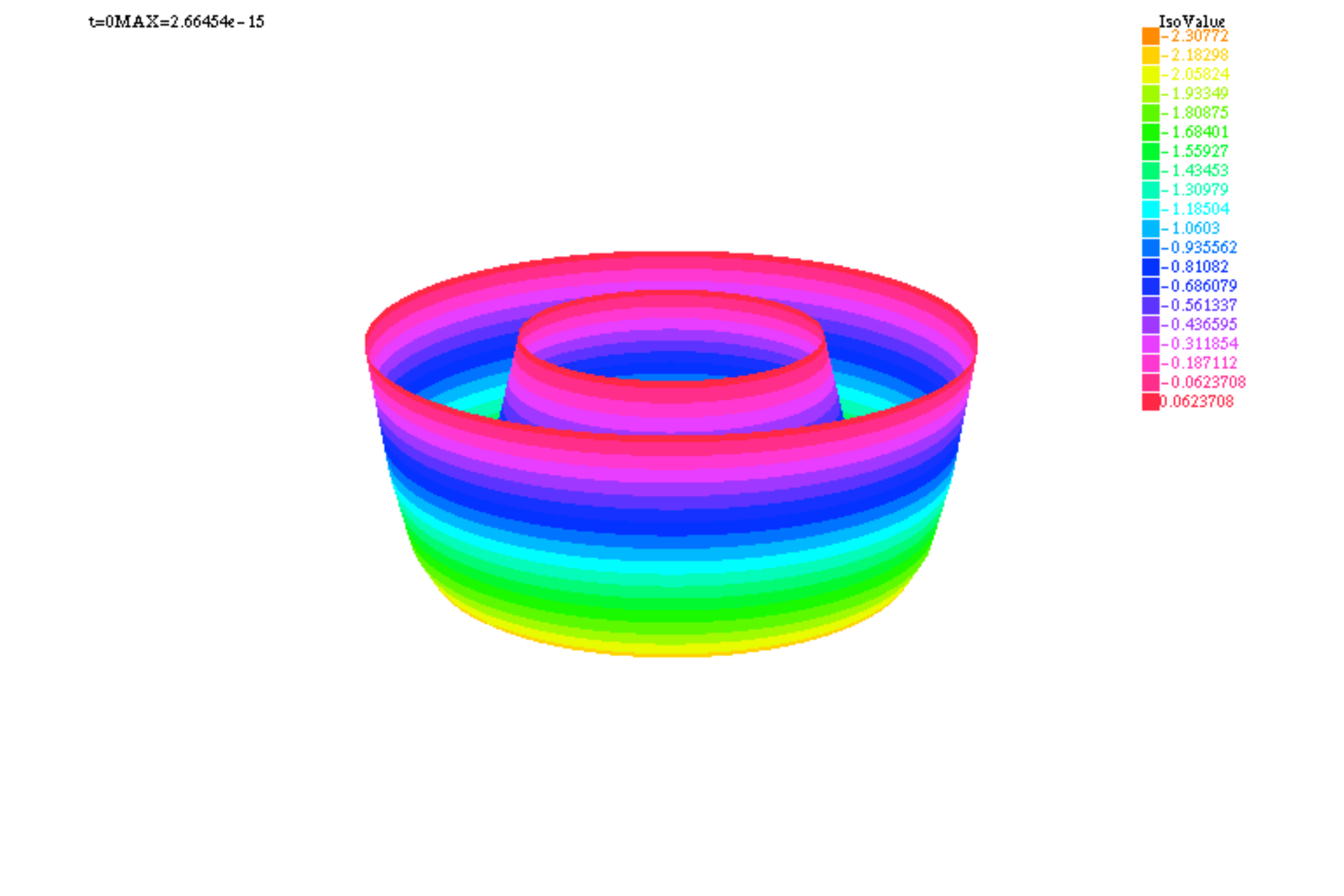}
\includegraphics[width=7.5cm, height=7.5cm]{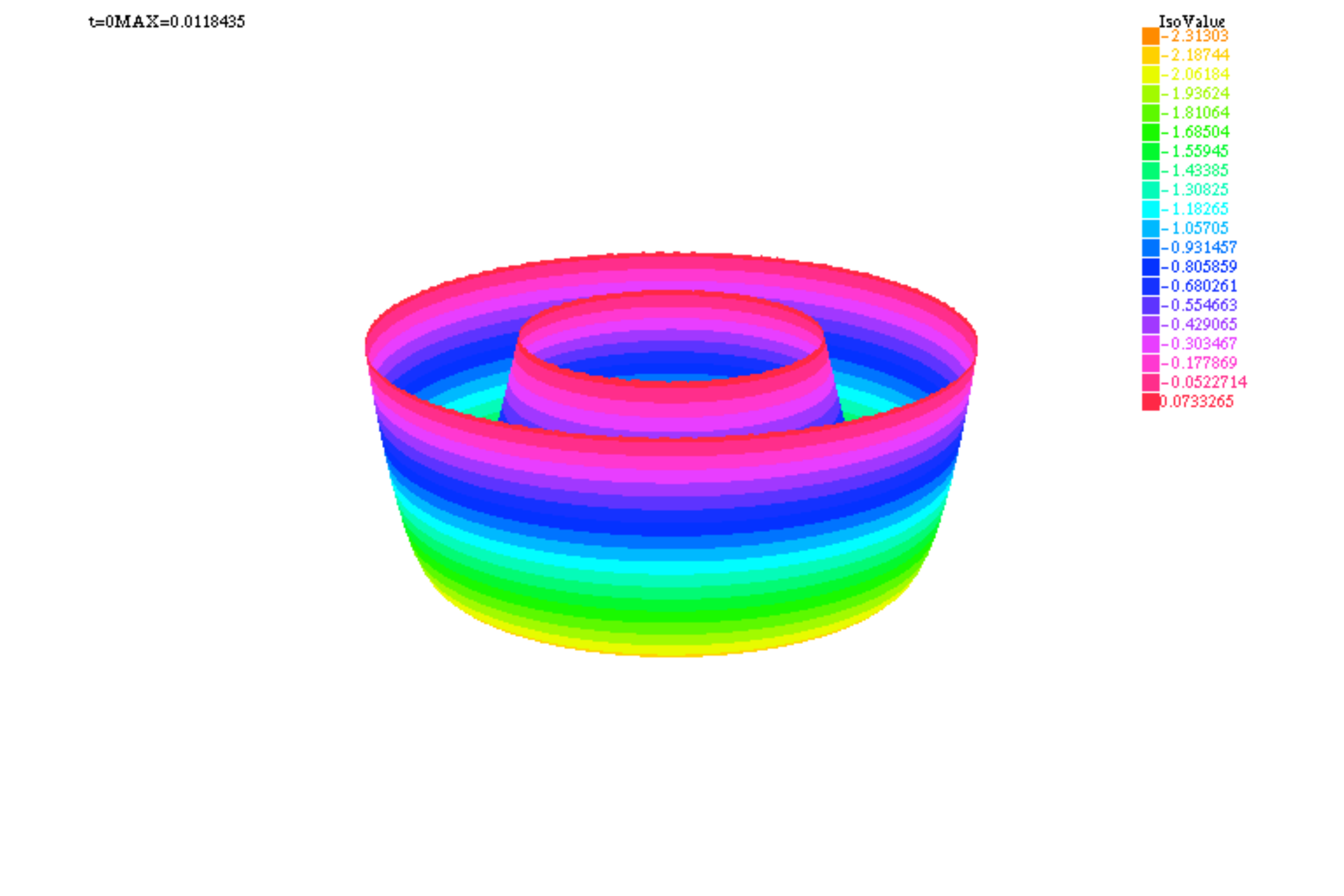}\\
\caption{Correction components for
$u(x,y)=(x^2+y^2-1)(x^2+y^2-4)$, 3d output and iso-values, $P_1$
elements} \label{exemple1zh}
\end{center}
\end{figure}
\begin{figure}[!htp]
\vskip -1.5cm
\begin{center}
\includegraphics[width=7.5cm, height=7.5cm]{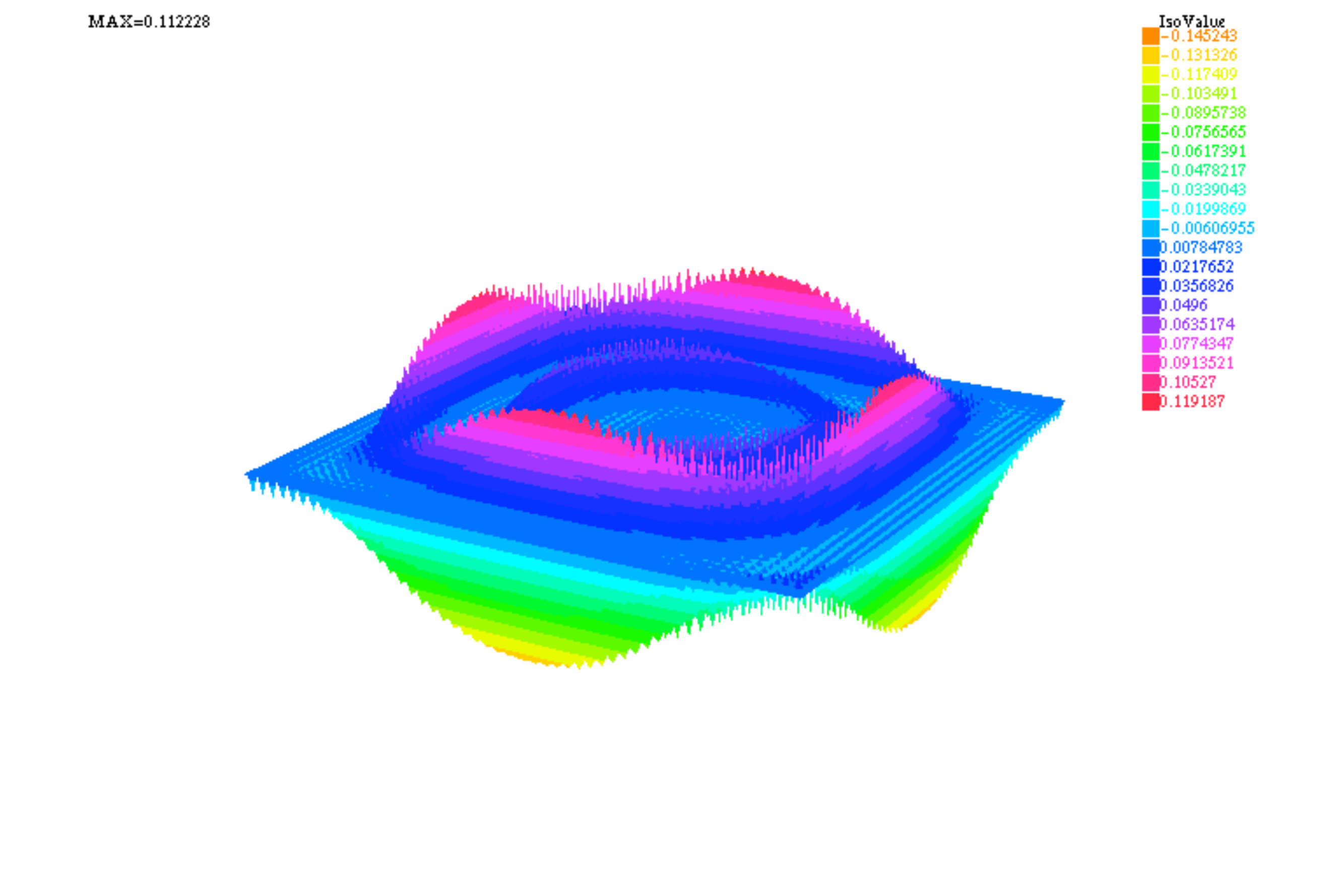}
\includegraphics[width=7.5cm, height=7.5cm]{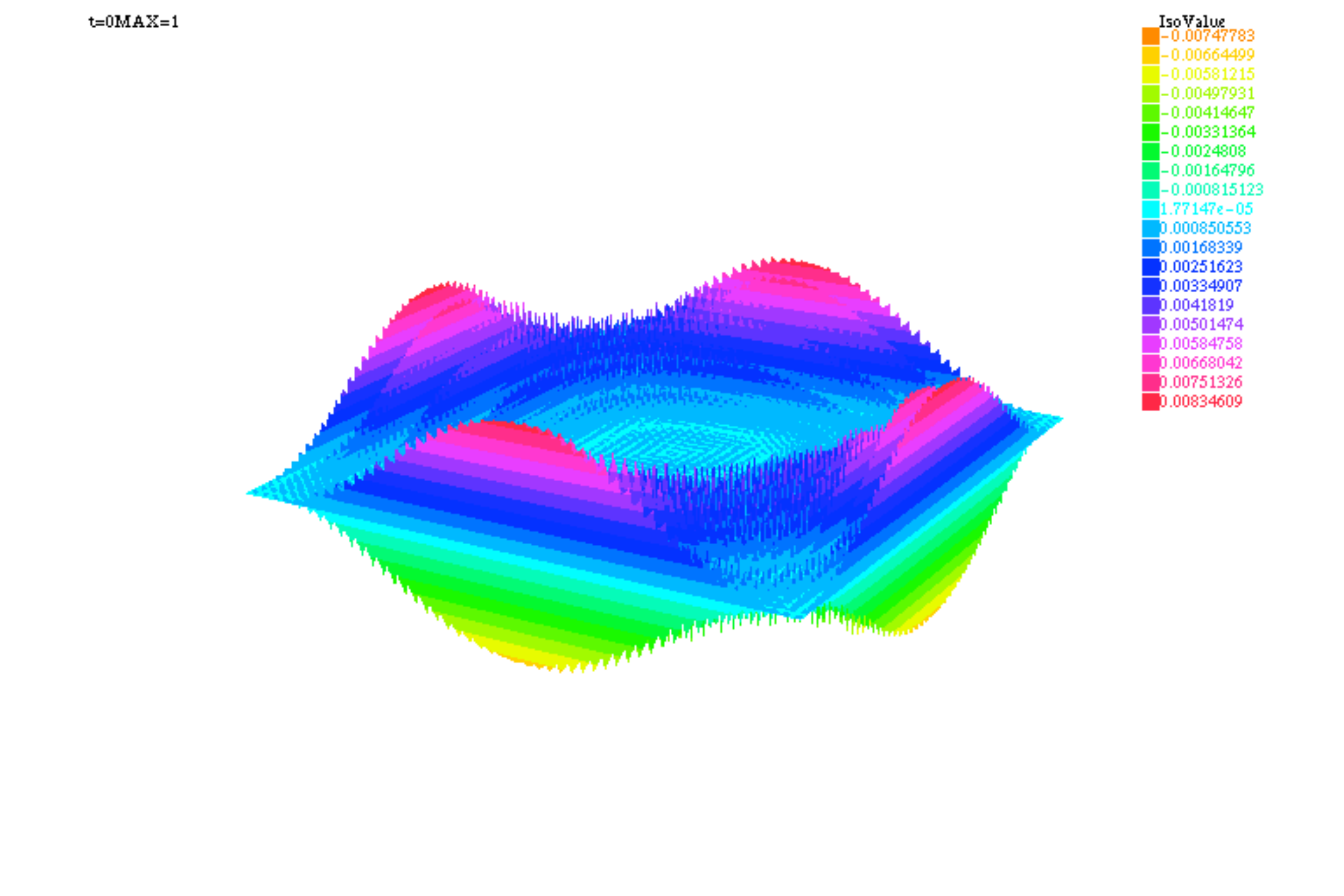}\\
\includegraphics[width=7.5cm, height=7.5cm]{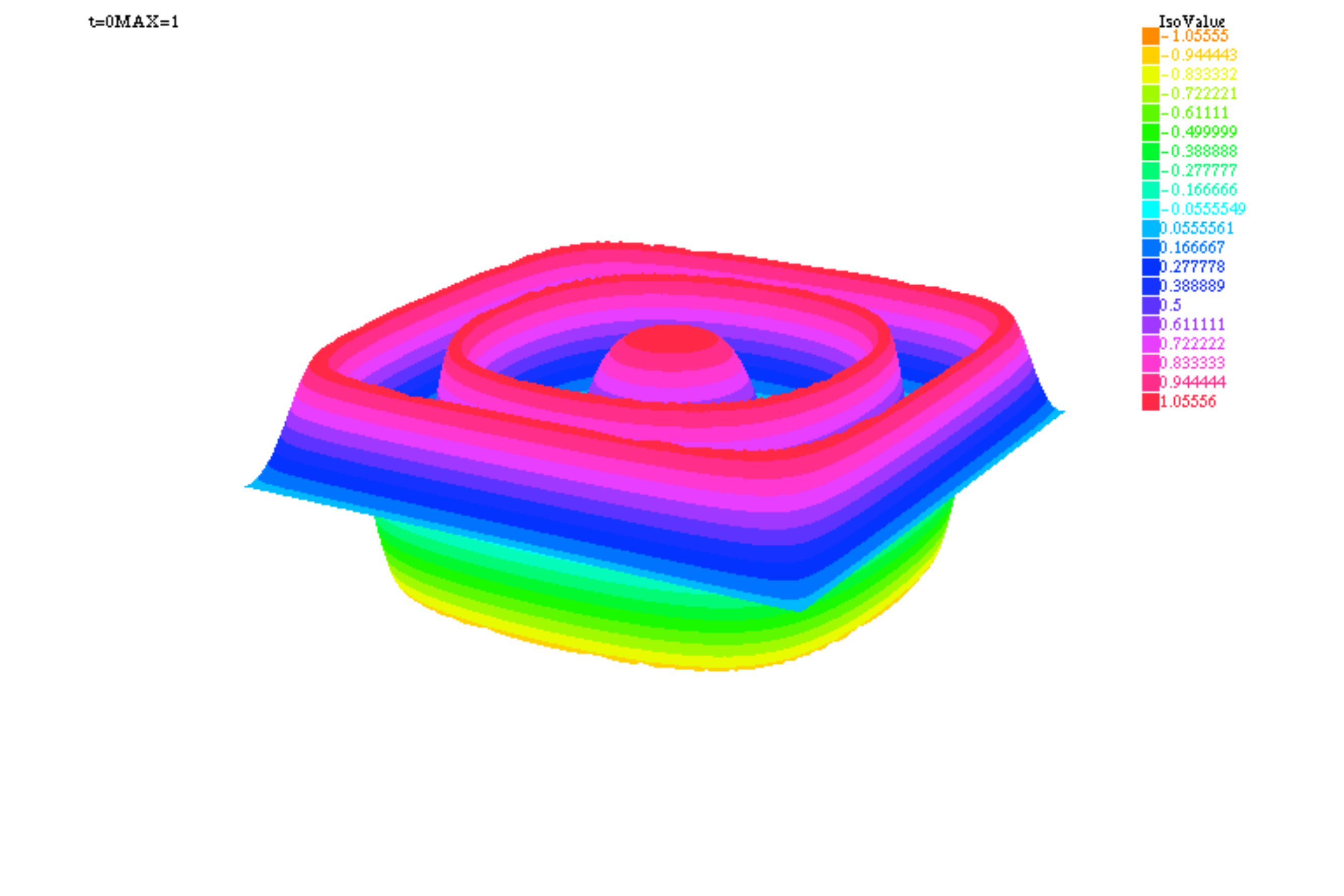}
\includegraphics[width=7.5cm, height=7.5cm]{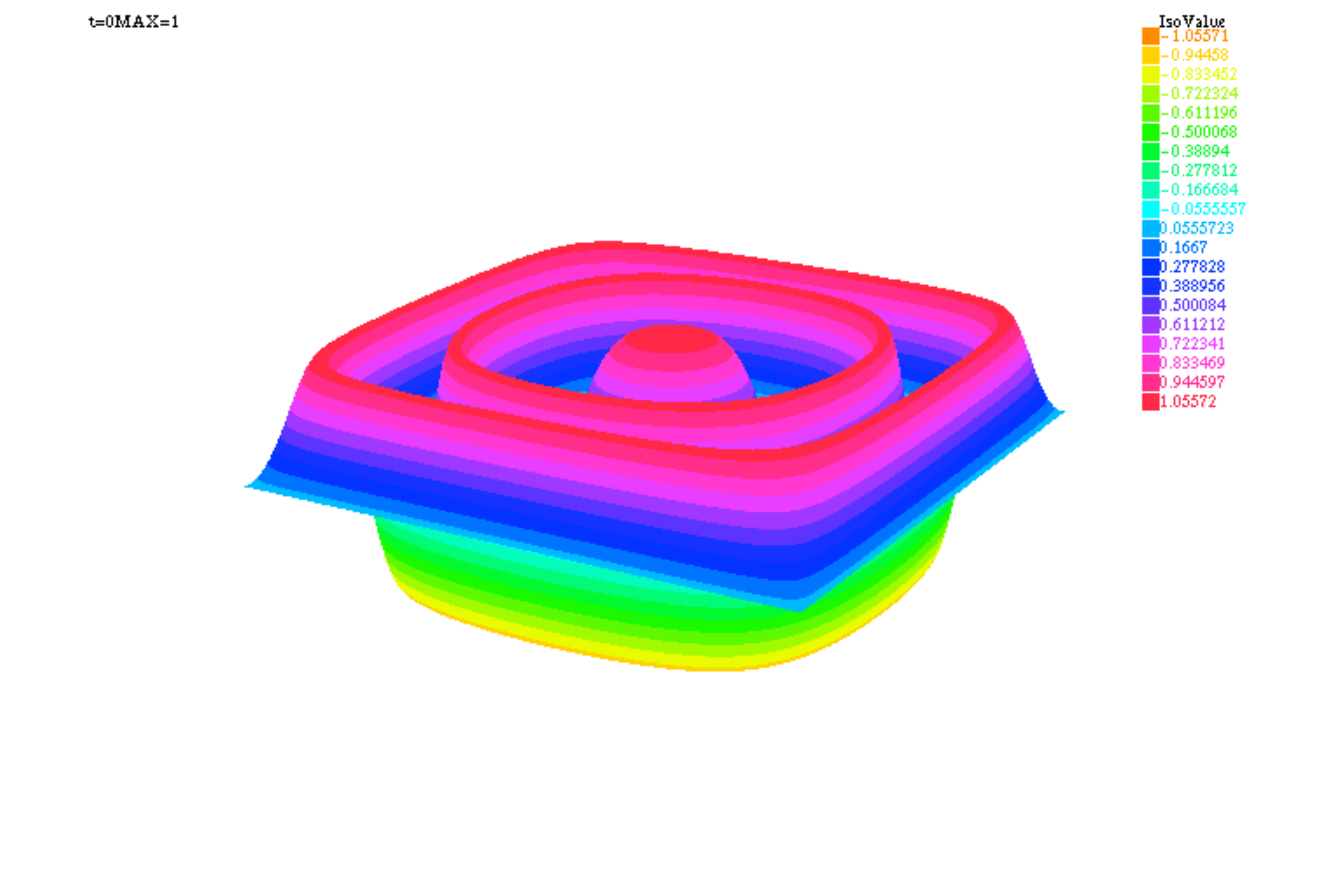}\\
\caption{Correction components for $u(x,y)=\sin(72x(1-x)y(1-y))$,
3d output and isovalues, $P_1$ elements} \label{exemple2zh}
\end{center}
\end{figure}
\begin{remark}
The same can be done with Neumann Boundary Conditions (B.C.) that
are the usual B.C. for the Allen-Cahn equation on which we will
concentrate now.
\end{remark}
In the previous figures we have illustrated the effect of the scale separation in space: the fluctuant part of the function, $z$,
is small in amplitude as respected to the one of the original function. This is observed when considering $P_1$ as well as $P_2$ elements. We observe also that the $z$-component exhibit high oscillations that are characteristic to the high frequency part of a function. We now quantify this property.\\
We generate the approximations of the
eigenfunction in the FEM space by solving numerically eigenvalues
problem,
\begin{eqnarray}
\int_{\Omega}u_hv_hdx + \int_{\Omega}\nabla u_h\nabla
v_hdx=\lambda \int_{\Omega}u_hv_hdx & \forall v_h \in V_h,
\end{eqnarray}
which is equivalent to find the eigen-elements of
\begin{eqnarray}
A_h u=\lambda M_h u,
\end{eqnarray}
where $A_h$ and $M_h$ are respectively the stiffness and the mass
matrix on the FEM space $V_h$. Denote by $(w^{(i)},\lambda_i)$ the
eigenvectors associated to the eigenvalue $\lambda_i$, we compare
the first eigen-components of $U$ and those of $Z$ to point out,
that is
$$
 \int_{\Omega}Uw^{(i)}dx \mbox{ and }   \int_{\Omega}Zw^{(i)}dx.
$$
In Figures \ref{spect_P2} and \ref{spect_P1}, we have represented
the energy spectrum of $u(x,y)=\cos(44(1-x)xy(1-y))$ and of its
fluctuant component $z$, when using $P_2$ as well as $P_1$
elements. We observe that the low mode components of $z$ are
reduced
in an important way as respected to the ones of $u$ while the high modes components are less smoothed; hence $z$ carries the high frequencies.\\
\begin{figure}[h!]
\vskip -0.8cm
\begin{center}
\includegraphics[width=14.cm, height=6.2cm]{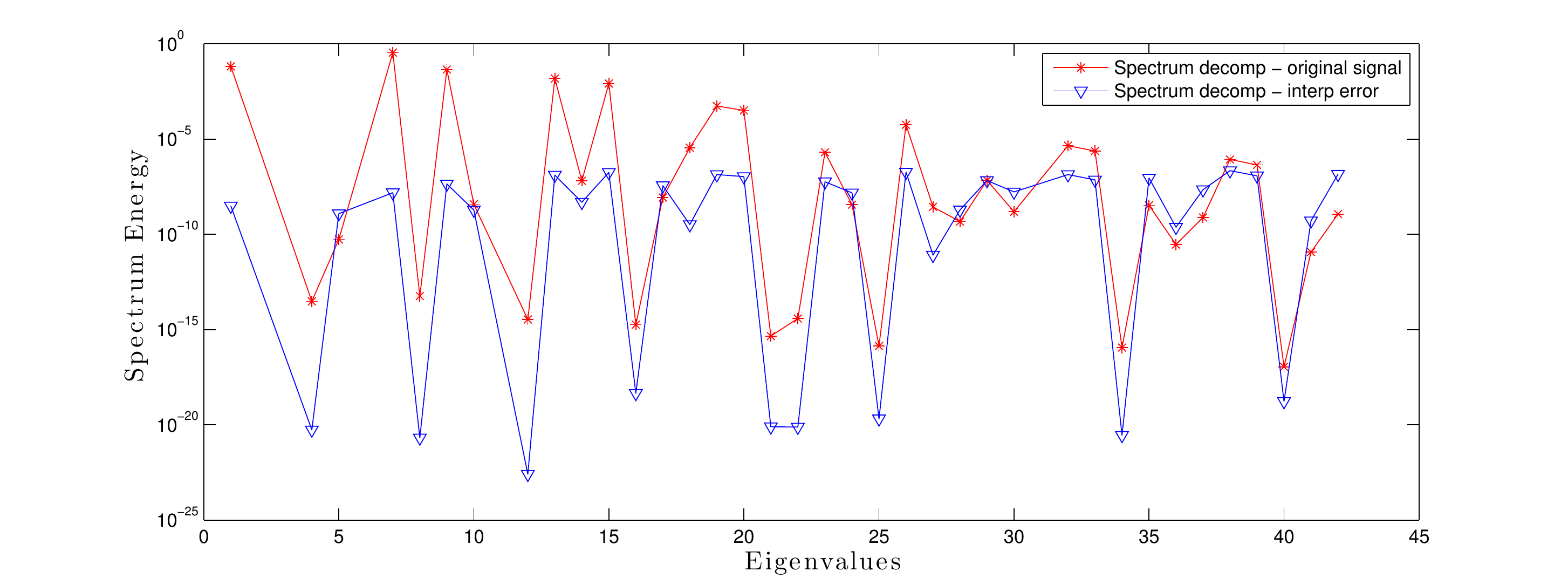}
\end{center}
\caption{ $P2$ FEM, 
$u(x,y)=\cos(44(1-x)xy(1-y))$, $\mbox{dim}(V_h)=1681, \  \mbox{dim}(W_H)=441$} \label{En_spect1}
\label{spect_P2}
\end{figure}
\newpage
\begin{figure}[h!]
\begin{center}
\includegraphics[width=14.cm, height=6.2cm]{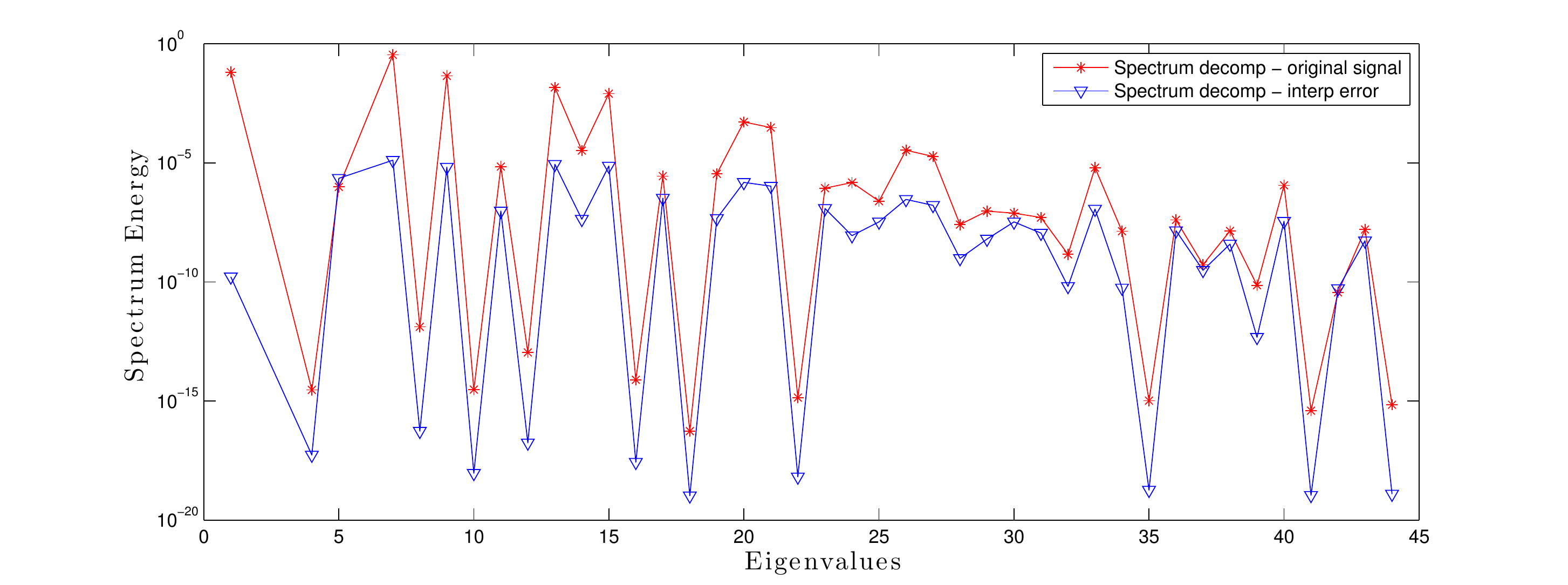}
\end{center}
\caption{ $P1$ FEM, 
$u(x,y)=\cos(44(1-x)xy(1-y))$, $\mbox{dim}(V_h)=1681, \  \mbox{dim}(W_H)=441$}\label{En_spect2}
\label{spect_P1}
\end{figure}
\subsection{The bi-grid framework}
\begin{center}
\begin{minipage}[H]{12cm}
  \begin{algorithm}[H]
    \caption{Scheme 4: Bi-grid scheme }\label{Gen_bi_grid}
    \begin{algorithmic}[1]
     \For{$k=0,1, \cdots$}\\
           \State {\bf Solve in $W_H$ }($\ds\frac{u_H^{k+1}-u_H^k}{\Delta t},\psi_H) + (\nabla
u_H^{n+1},\nabla \psi_H)+(DF(u^{k+1}_H, u^{k}_H),\psi_H) =0, \forall \psi_H \in V_H$
\State {\bf Solve in $V_h$ } $(\tilde{u}_h^{k+1},\phi_h)-(u_H^{k+1},\phi_h)=0, \forall \phi_h \in V_h$
             \State {\bf Solve}
             $$
             \begin{array}{ll}
\medskip
(1+\tau \Delta t)(z_h^{k+1}-z_h^k,\phi_h)
+\Delta t (\nabla z_h^{k+1},\nabla \phi_h)
+\Delta t (\nabla \tilde{u}_h^{k+1},\nabla \phi_h)\\
\medskip
+(\tilde{u}_h^{k+1}-\tilde{u}_h^k,\phi_h) +\Delta t ({\tilde
f}(u_h^k,{\tilde u}_h^{k+1},z_h^{k+1}),\phi_h)=0,\forall
\phi_h \in V_h\\
\end{array}
$$
\State {\bf Set}   $u^{k+1}= \tilde{u}_h^{k+1}+z_h^{k+1}$
            \EndFor
    \end{algorithmic}
    \end{algorithm}
\end{minipage}
\end{center}
Here ${\tilde f}(u^k,{\tilde u}^{k+1}_h,z_h^{k+1})$ is an approximation of $f(u^{k+1}_h)$, e.g. ${\tilde f}(u^k,{\tilde u}^{k+1}_h,z_h^{k+1})=f(u^{k}_h)$.
For this choice of ${\tilde f}$ we have an important property:
the high mode stabilization makes the bi-grid scheme consistent with the computation of the steady states. More precisely, we have
\begin{proposition}
Assume, for fixed $h$ and $H$, that there exists a unique pair of
elements ${\bar u}_H\in W_H$ and ${\bar u}_h\in V_h$ such that
$$
(\nabla {\bar u}_H,\nabla \psi_H)+(f({\bar u}_H),\psi_H)=0,
\forall \psi_H \in W_H \mbox{ and } (\nabla {\bar u}_h,\nabla
\phi_h)+(f({\bar u}_h),\phi_h)=0, \forall \phi_h \in V_h.
$$
Assume that $\lim_{k\rightarrow +\infty} u_H^k={\bar u}_H$  and
that $z^k_h$ is convergent to ${\bar z}_h$. Then
$$
 \mbox{ and }
\lim_{k\rightarrow +\infty} u_h^k={\bar u}_h.
$$
\end{proposition}
\begin{proof}
To establish the consistency, we show that $u^k_h=z^k_h+{\tilde u}^k_h$ converges to ${\bar u}_h$.\\
By continuity of the prolongation ${\cal P}$, we have
$$
{\cal P}(u^n_H)={\tilde u}_h^k\rightarrow  {\cal P}({\bar u}_H)={\tilde {\bar u}}_h, \mbox{ as } k\rightarrow +\infty .
$$
Taking the limit in the correction step of the scheme, we get,
after the usual simplifications
$$
 (\nabla z_h,\nabla \phi_h)+(\nabla \tilde{u}_h,\nabla \phi_h)
+(f(z_h+{\tilde u}_h),\phi_h)=0,\forall
\phi_h \in V_h.
$$
Letting $w_h=z_h+{\tilde u}_h$, we find
$$
(\nabla w_h,\nabla \phi_h)+(f(w_h),\phi_h)=0, \forall
\phi_h \in V_h.
$$
By identification, $w_h=z_h+{\tilde u}_h={\bar u}_h$.
\end{proof}
\section{A big-grid method for Allen-Cahn Equation}
As presented above, the bi-grid scheme is based on an implicit
(stable) scheme applied on the coarse space $W_H$ and on a
simplified semi-implicit scheme on the fine space $V_h$, for the
computation of the correction (fluctuant) term $z$. The scheme on
$W_H$ is considered to be the reference scheme. Its implementation
necessitates the numerical solution of a fixed point problem at
each time step. We present hereafter a way to overcome the
artificial instability carried by the use of the classical Picard
iterates. The new nonlinear iterations will be implemented to
define the effective reference scheme when applied to $V_h$ and to
which we will compare the bi-grid schemes.
\subsection{A solution to an artificial instability problem for a one-grid scheme}
\noindent The implementation of the scheme (\ref{AC_implicit1})
needs a fixed point problem to be solved at each time step. Let
$M_h$ and $A_h$ be the mass and the stiffness matrices
respectively. If we set
$$\phi(v,u^k)=\Big(M_h+\Delta t A_h\Big)^{-1}\Big\{u^k-\ds\frac{\Delta t}{\epsilon^2}DF(u^k,v)\Big\} ,$$\\
\noindent then the time marching scheme reduces to solve the fixed
point problem
\begin{equation}\label{pointfixe}
v=\phi(v,u^k)
\end{equation}
at the n$^{th}$ time step. In practice, the convergence of the
Picard iterates is obtained by taking only very small values of
$\Delta t$, typically $\Delta t\simeq 10^{-4}$. This is dramatic
since we are looking to the long time numerical behavior of the
solution. Anyway, this effective restriction on the time step is
really artificial because the scheme is supposed to be
unconditionally stable. For this reason, as in \cite{AACDG} (in
the Nonlinear Schr\"{o}dinger Equation case), we apply here the
extrapolation of the fixed point to compute $u^{k+1}$ from $u^k$
and we propose to solve (\ref{pointfixe}) by accelerating the
(Picard) sequence
\begin{equation}\label{PicardSeq}
\left.
\begin{array}{ll}
\medskip
v^0=u^n,\\
\medskip
\text{for}\;\; m=0,\ldots\\
\medskip
v^{(m+1)}=\phi(v^{(m)},u^k),
\end{array}
\right.
\end{equation}
\noindent enhancing in that the stability region, allowing then to
take larger values of $\Delta t$. To this end, we use the
$\Delta^\kappa$ acceleration procedure, see \cite{BrezJP}.
In two words, the $\Delta^\kappa$ procedure
consists in replacing the Picard iterates by
\begin{equation}\label{MarechalSeq}
\left.
\begin{array}{ll}
\medskip
v^0=u^k,\\
\medskip
\text{for}\;\; m=0,\ldots\\
\medskip
v^{(m+1)}=v^{(m)}-(-1)^\kappa\alpha_m^\kappa\Delta_\phi^k v^{(m)};
\end{array}
\right.
\end{equation}
\noindent where $\Delta_\phi^\kappa v^{(m)}=\ds\sum_{j=0}^\kappa
C_j^\kappa(-1)^{\kappa-j}\phi^{(j)}(v^{(m)},u^k),\;C_j^\kappa=\ds\frac{\kappa!}{j!(\kappa-j)!}$
is the binomial c$\oe$fficient and $\phi^{(j)}$ denotes the
$j^{th}$ composition of $\phi$ with itself. We have
\begin{equation}\label{alpha}
\alpha_m^\kappa=(-1)^\kappa\ds\frac{<\Delta_\phi^1
v^{(m)},\Delta_\phi^{\kappa+1}v^{(m)}>}{<\Delta_\phi^{\kappa+1}
v^{(m)},\Delta_\phi^{\kappa+1} v^{(m)}>},
\end{equation}
\vspace{-0.1cm} \noindent where $<.,.>$ denotes the euclidean
scalar product in $\R^n$, see \cite{BrezJP}. These acceleration
procedures have been applied with the $\Delta^1$ (Lemar\'echal's method\cite{Lemarechal} corresponding to
$\kappa=1$); we can take $\Delta t=10^{-2}$ and the scheme
(\ref{AC_implicit1}) is still stable.
 A comparison of the energy curves shows a digital
convergence by varying the number of nodes on the edge of the unit
square $N$  to generate the mesh ${\cal T}_h$, for $\epsilon=0.03$ and Lemar\'echal's acceleration.
%
\vspace{-0.5cm}
\begin{center}
\begin{figure}[!hp]
\begin{center}
\includegraphics[width=7.2cm,height=4.3cm]{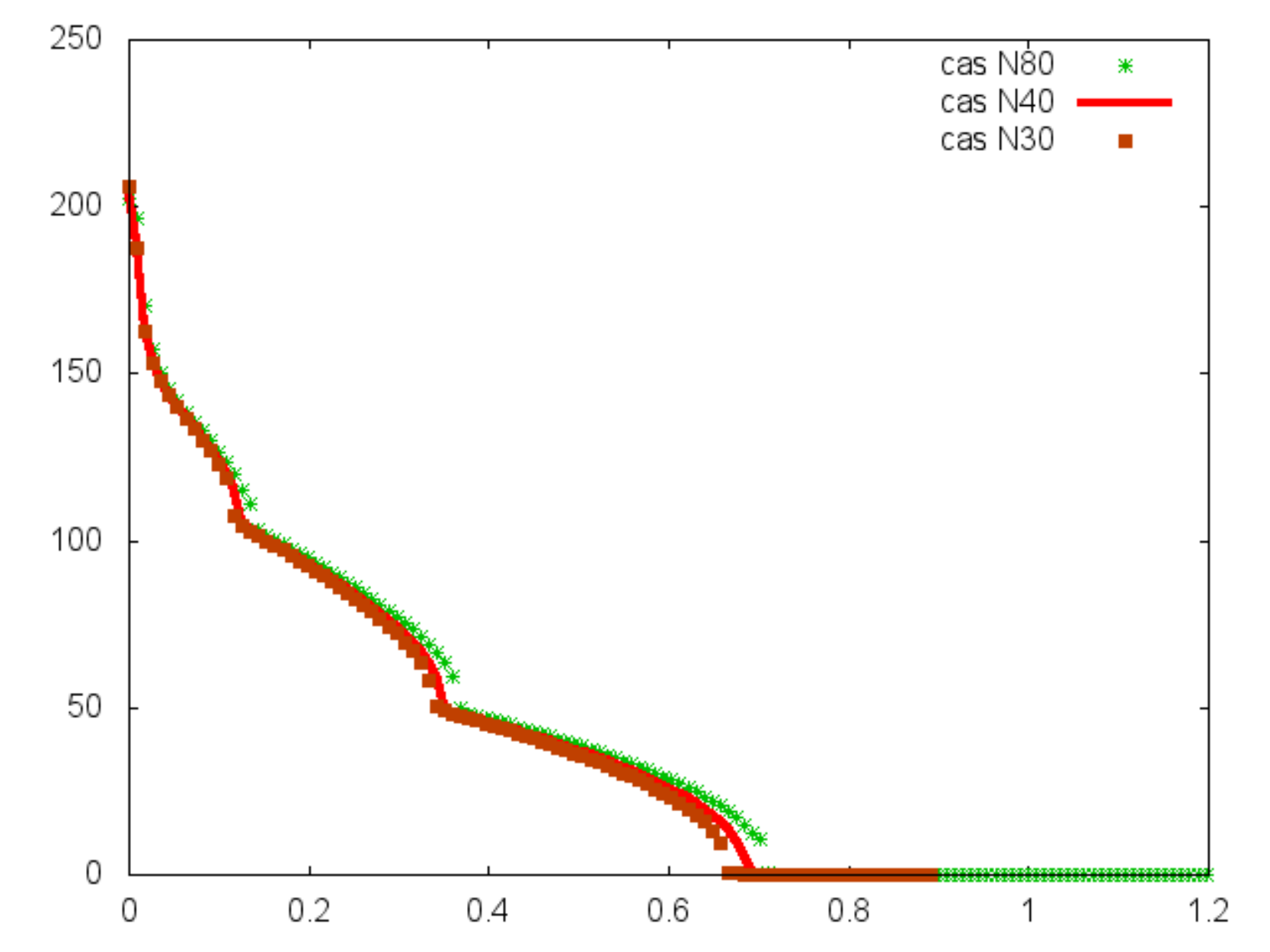}
\end{center}
\vspace{-0.2cm}
\caption{Energy curves for $N=30 \ (\mbox{dim}(V_h)=961),N=40  \ (\mbox{dim}(V_h)=1681), N=80 \ (\mbox{dim}(V_h)=6561)$ on the unit square with
$P_1$ element, $\Delta t=0.009$ and Lemar\'echal method.}
\end{figure}
\end{center}
\newpage
\begin{center}
\begin{figure}[!hp]
\begin{center}
\includegraphics[width=7.2cm,height=4.3cm]{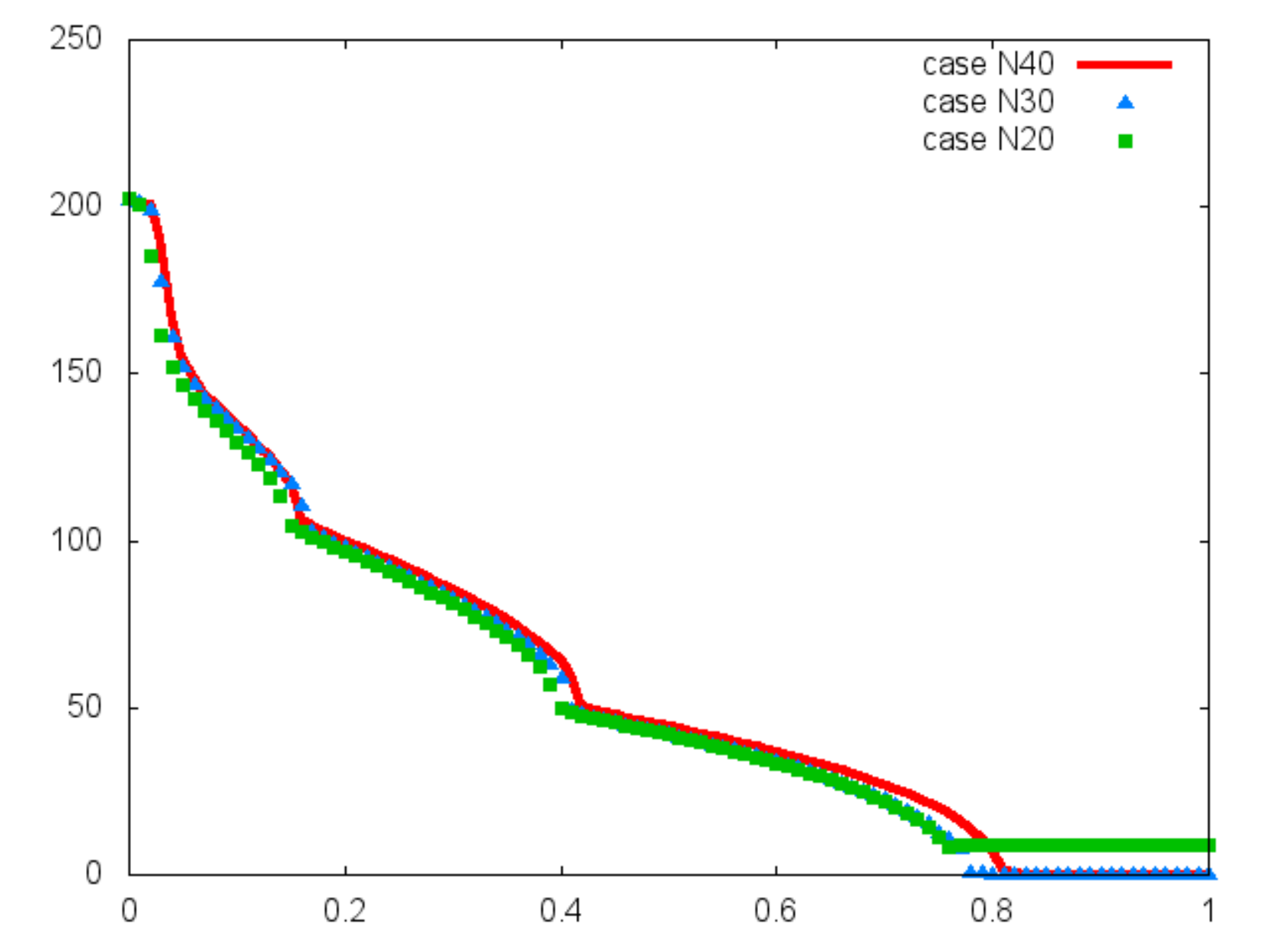}
\end{center}
\vspace{-0.2cm}
\caption{Energy curves for $N=20 \ (\mbox{dim}(V_h)=1681),N=30  \
(\mbox{dim}(V_h)=3721), N=40 \ (\mbox{dim}(V_h)=6561)$ on the unit
square with $P_2$ element, $\Delta t=10^{-2}$ and Lemar\'echal
method.}
\end{figure}
\end{center}
%
\vspace{-0.7cm}
%
 The efficiency of this method appears in reducing the
number of internal iterations required to reach a final time $T$. We
present the CPU computation time and the number of internal iterations
following the numerical implementation done on the unit square for
$T=0.4$ and the maximum time step respectively for Lemar\'echal
and Picard method.
%
$$
\begin{tabular}{|l|l|l|l|l|l|}
\hline
Case P2&Picard & $\Delta t=3\times10^{-4}$ & T=0.4 & Nb iter=34266 & CPU=2851.62s\\
\cline{2-6}
N=20&Lemarechal &$\Delta t=5\times10^{-2}$ & T=0.4& Nb iter=360 & CPU=73.595s\\
\hline
Case P2&Picard & $\Delta t=3\times10^{-4}$ & T=0.4 & Nb iter=34796 & CPU=11639.8s\\
\cline{2-6}
N=40&Lemarechal &$\Delta t=10^{-2}$ & T=0.4& Nb iter=1600 & CPU=1598.8s\\
\hline
%
%
Case P2&Picard & $\Delta t=3\times10^{-4}$ & T=0.4 & Nb iter=35142 & CPU=47686.2s\\
\cline{2-6}
N=80&Lemarechal &$\Delta t=10^{-2}$ & T=0.4& Nb iter=1599 & CPU=11250.5s\\
%
\hline
\end{tabular}
$$

\subsubsection{Numerical results}
\vspace{-0.3cm} \noindent In order to compare the two fixed point
methods, we present below the evolution curve of the functional
energy over time and some numerical results for the
unconditionally stable scheme (\ref{AC_implicit1}) by using the
mesh of the unit square with $N= 40$, the initial condition
$u_{0}=\cos(4\pi x) \cos(4\pi y)$, the interfacial width
$\epsilon=0.03$, $ \Delta t=10^{-4}$  and the finite element
$P_2$.\\
\begin{figure}[!hp]
\begin{center}
\includegraphics[width=7.2cm, height=4.3cm]{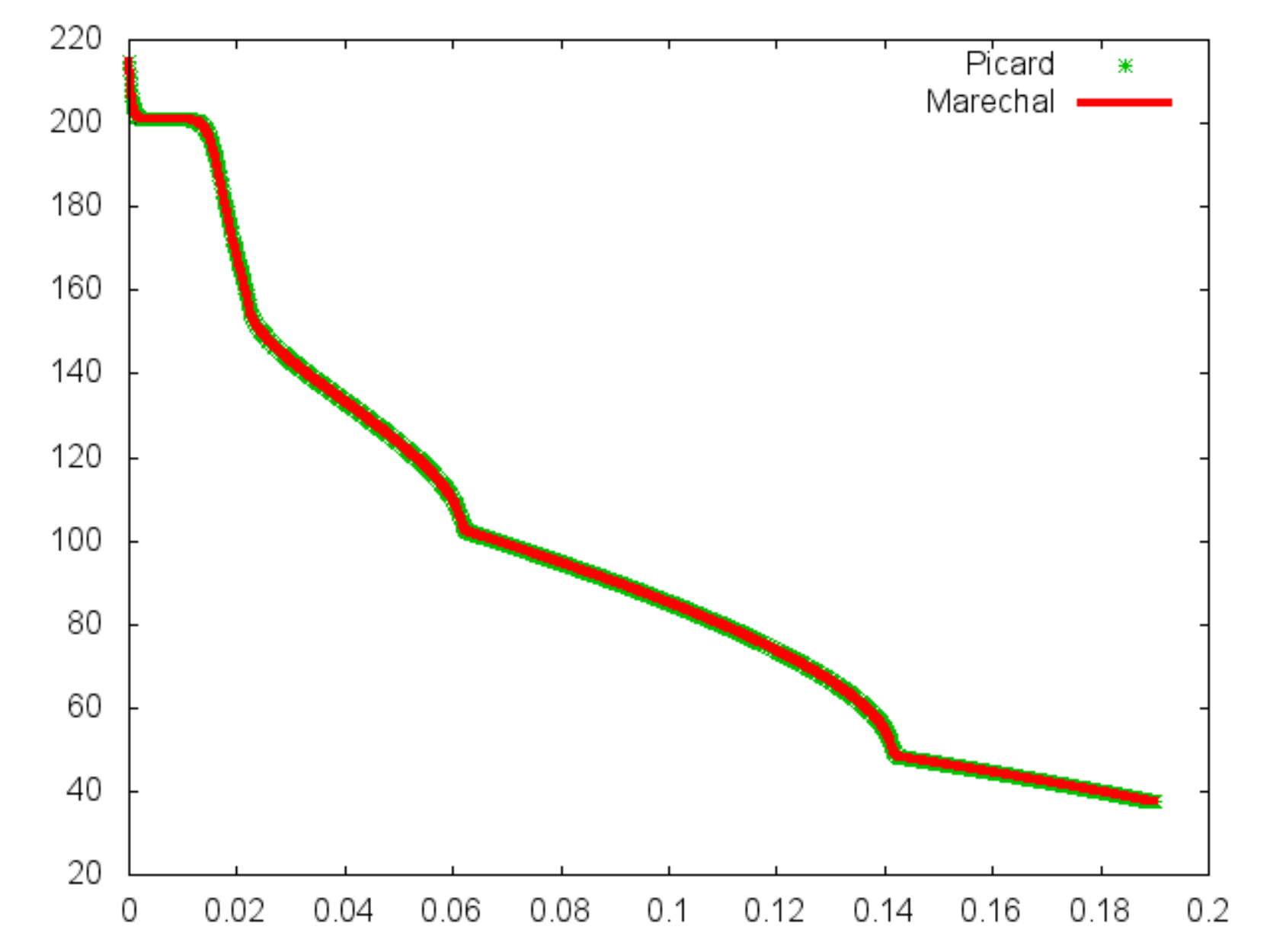}
\end{center}
%
\caption{Energy curves for Picard and for Lemar\'echal.}
\end{figure}
\newpage
\vspace{-0.5cm}
\begin{figure}[!hp]
\begin{center}
\includegraphics[width=5.5cm, height=4cm]{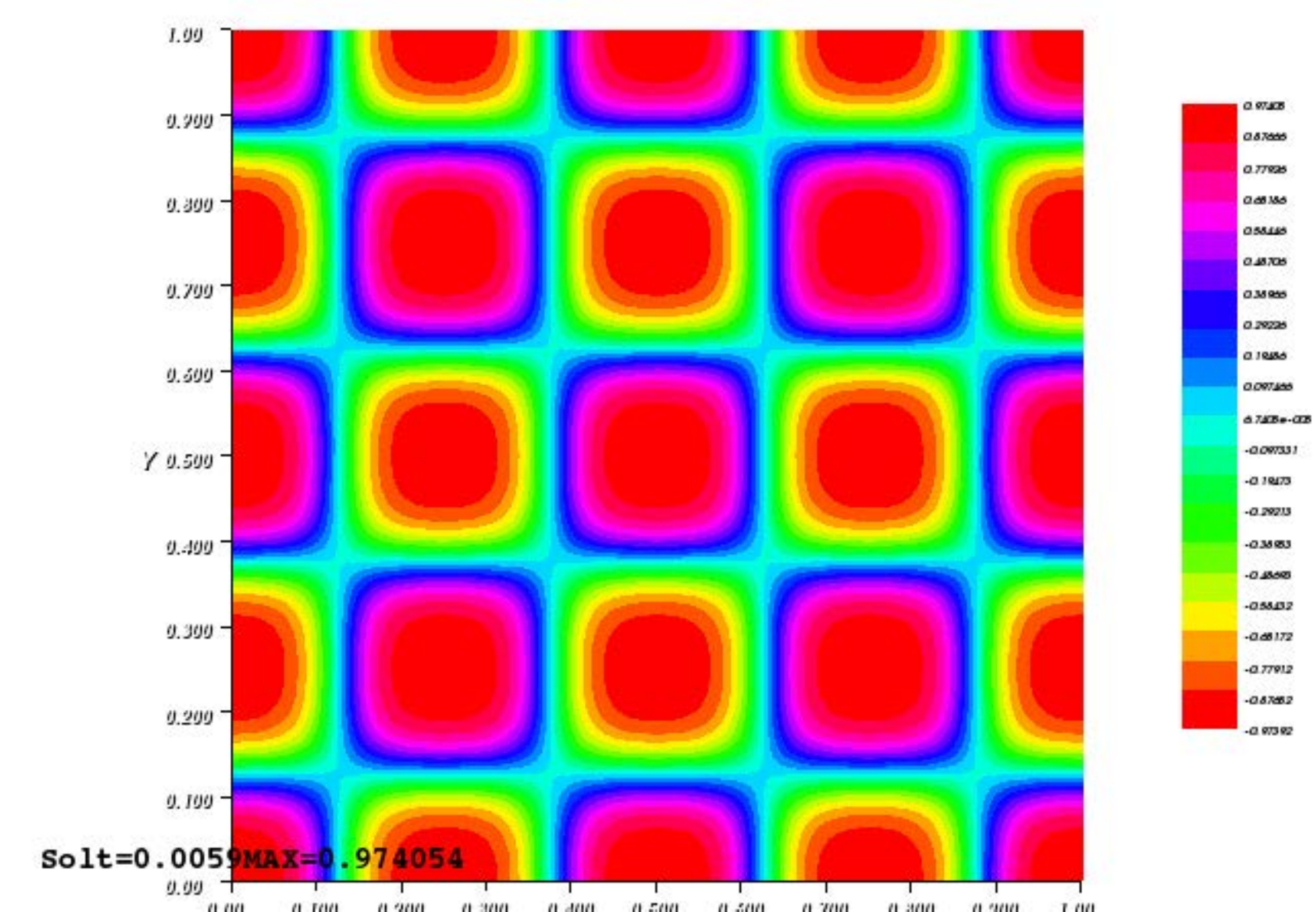}
\includegraphics[width=5.5cm, height=4cm]{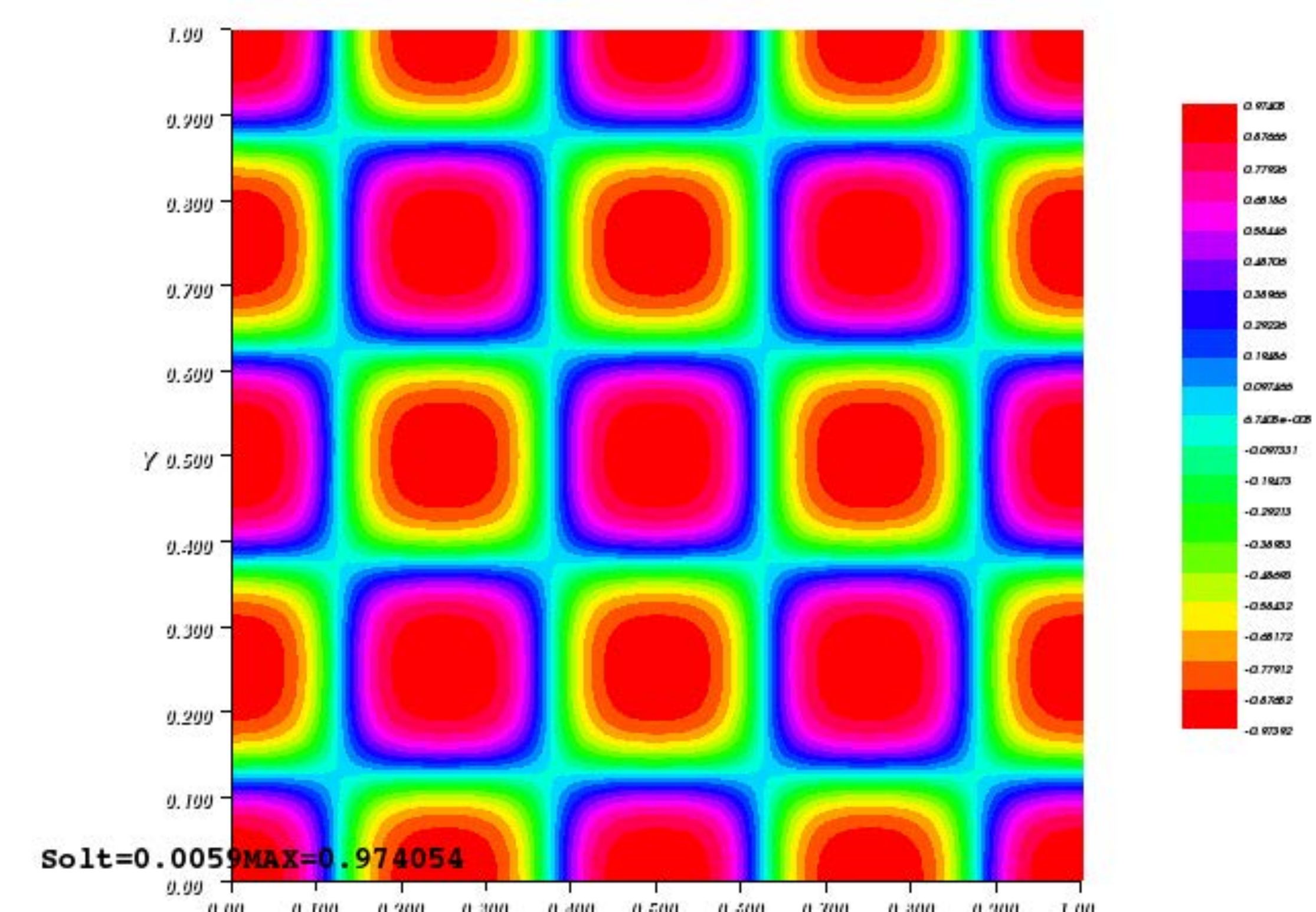}
\end{center}
\caption{Picard (left) and
Lemar\'echal (Right) at $t=0.0059$.}
\end{figure}

%
\vspace{-0.5cm}

\begin{figure}[!hp]
\begin{center}
\includegraphics[width=5.5cm, height=4cm]{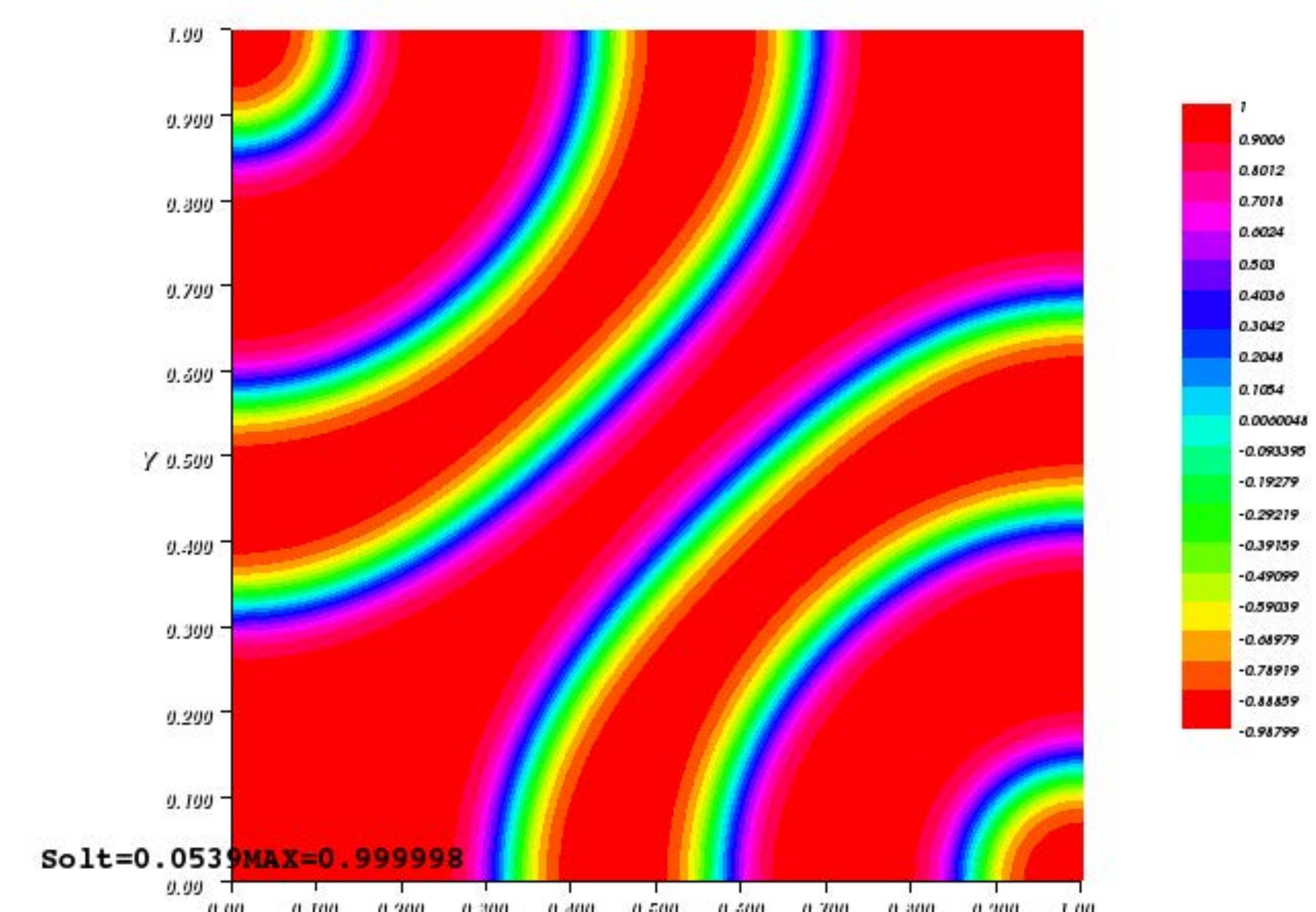}
\includegraphics[width=5.5cm, height=4cm]{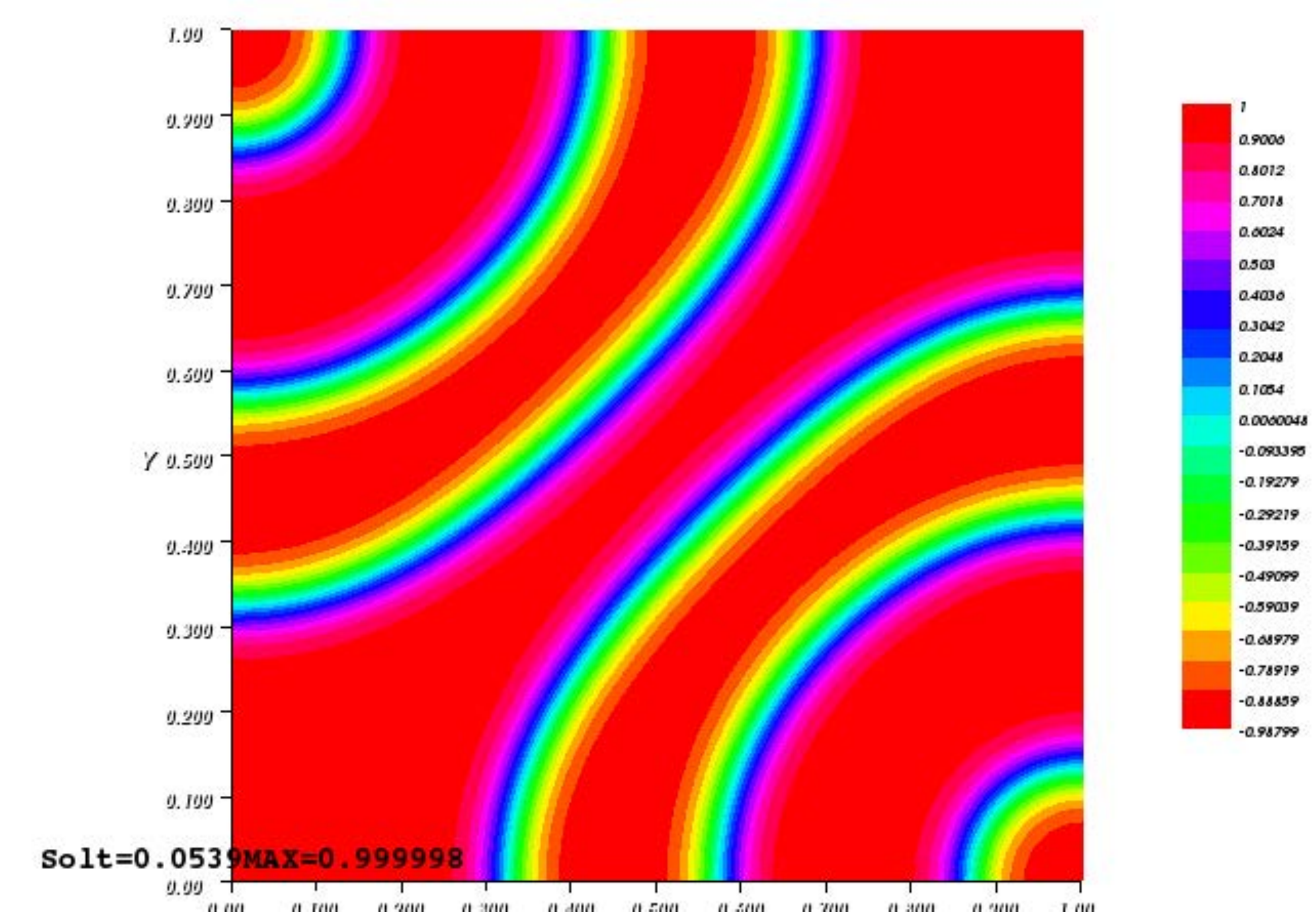}
\end{center}
\caption{Picard (left) and
Lemar\'echal (Right) at $t=0.0539$.}
\end{figure}
%

%

\begin{figure}[h!]
\begin{center}
\includegraphics[width=5.5cm, height=4cm]{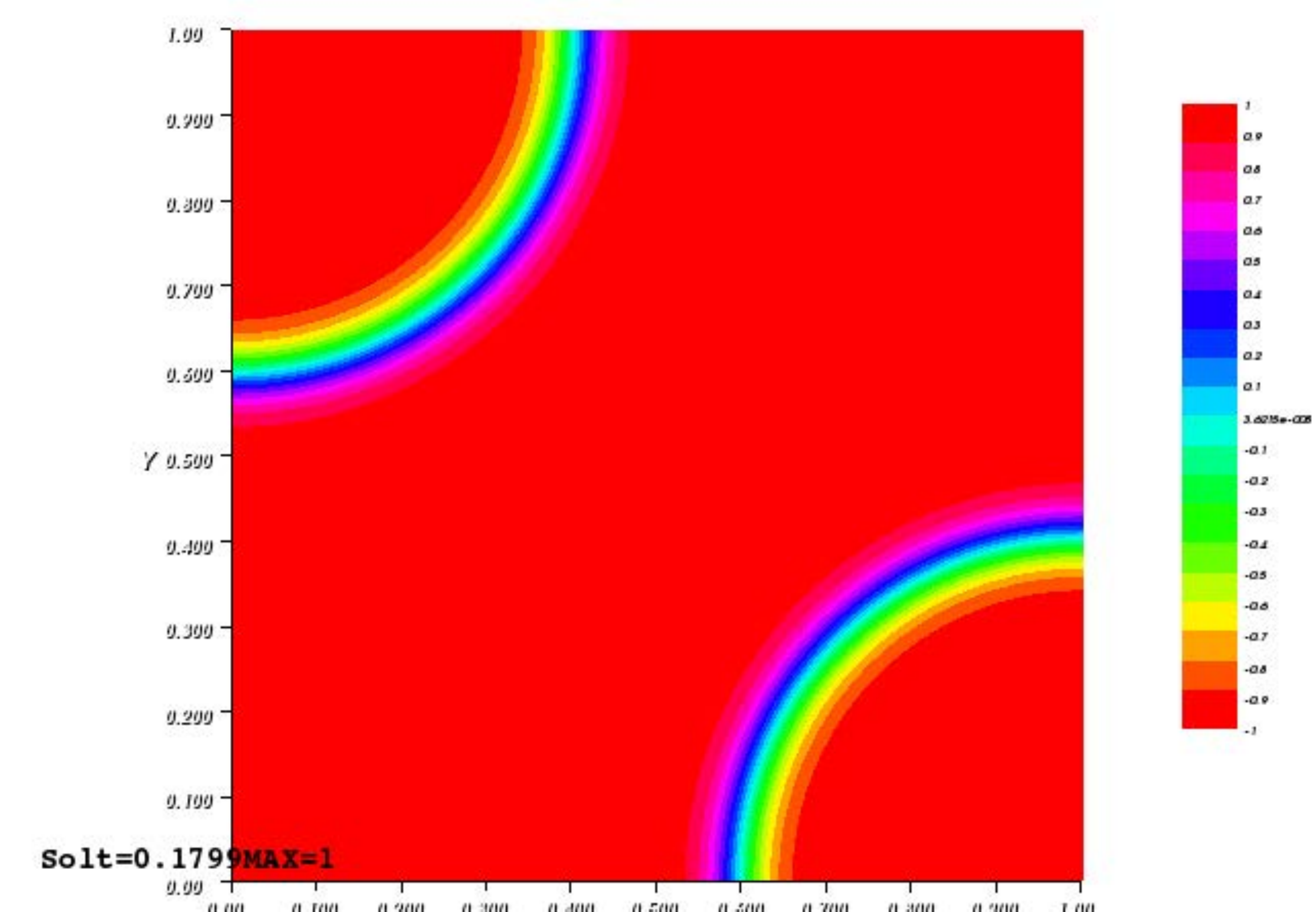}
\includegraphics[width=5.5cm, height=4cm]{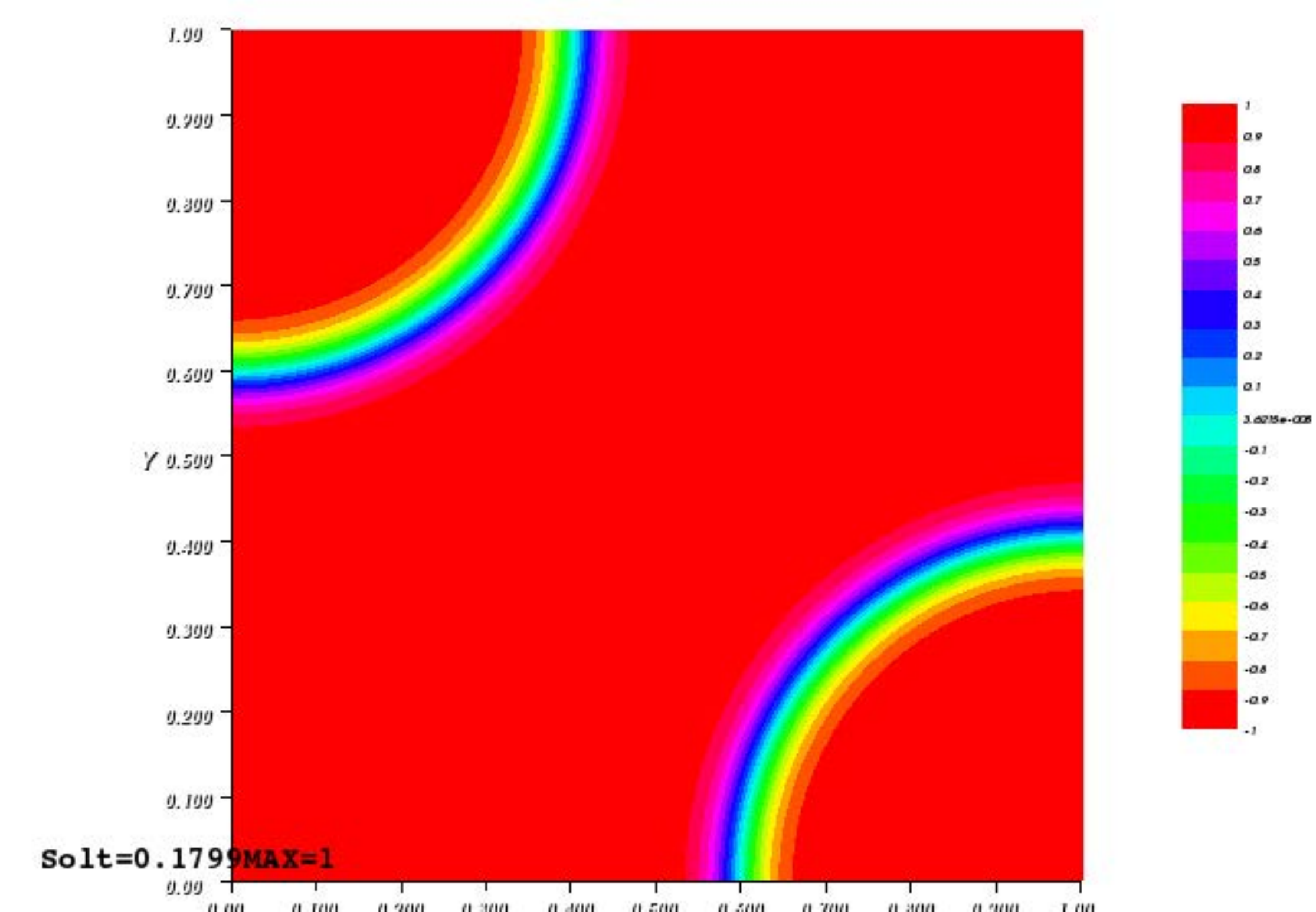}
\end{center}
%
\caption{Picard (left) and Lemar\'echal (right) at $t=0.1799$.}
\end{figure}
\subsection{How to choose $W_H$ and $V_h$}
A central question for the bi-grid method is the choice of the two approximation spaces $W_H$ and $V_h$. We have to balance two criteria:
\begin{itemize}
\item the CPU reduction capabilities brought by the bi-grid
scheme: the most important part of the computations is realized in
the coarse space (nonlinear iterations). This is directelty related
to the ratio of the dimensions of $W_H$ and $V_h$, that we denote
by $DR=\Frac{\mbox{dim($W_H$)}}{\mbox{dim($V_h$)}}$.
\item the scale separation in frequency that allows to make a
correction through a simplified scheme on $V_h$. It appears that a
good indicator is that the fine space correction term $z$ is
 in norm, this means that the approximation on the coarse
space is sufficiently correct to be an acceptable approximation on
the fine space once prolongated. According to the general case,
see Proposition \ref{z_estimates}, an indicator of the $L^2$ norm
of $z$ is controlled by the interpolation error in $W_H$.
\end{itemize}
Estimates on the relative step size of $W_H$ and $V_h$ in the case $W_H\subset V_h$ have been obtained in different contexts, see \cite{MarionXu,Xu} for bi-grid Nonlinear Galerkin (reaction-diffusion problem) and
\cite{AbbGirSay2} for Navier-Stokes time dependent equations.\\

Here the condition $W_H\subset V_h$ is not mandatory and we can then define a mixed finite element scheme,
 the compatibillty condition between the FEM is given in proposition \ref{graph_inf_sup}.\\

We give hereafter possible choices for $W_H$ and $V_h$ in several situations:
\begin{itemize}
\item  The case $W_H\subset V_h$
\begin{itemize}
\item When ${\cal T}_H={\cal T}_h$:
$$
W_H=\{v_h\in{\cal C}^{0}({\bar \Omega}), v_h|_{K} \in P_p, \forall K \in {\cal T}_h\},
\
V_h=\{v_h\in{\cal C}^{0}({\bar \Omega}), v_h|_{K} \in P_q, \forall K \in {\cal T}_h\},
$$
with $q>p$. For instance $p=1, q=2$ or $q=3$.\\
\\
Following Proposition \ref{z_estimates}, it is easy to see that
$$
\|u-\Pi_hu\|_{L^2(\Omega)}\le C h^{q+1} \|u \|_{H^{q+1}(\Omega)}
 \mbox{ and } \|u-\Pi_Hu\|_{L^2(\Omega)}\le C h^{p+1}\|u \|_{H^{p+1}(\Omega)},
$$
so the {\it a priori} optimal estimation for the prolongation error is $H^{p+1}$.
\item When ${\cal T}_H\subset {\cal T}_h$:
$$
W_H=\{v_h\in{\cal C}^{m}({\bar \Omega}), v_h|_{K} \in P_p, \forall K \in {\cal T}_h\},
\
V_h=\{v_h\in{\cal C}^{m}({\bar \Omega}), v_h|_{K} \in P_p, \forall K \in {\cal T}_h\}.
$$
Another way to build $W_H$ from $V_h$ is to apply a coarsening procedure, we refer the reader
e.g. to \cite{BankXu2}.
\end{itemize}
\item The case $W_H\not \subset V_h$: the {\it a priori} estimates
given by Proposition \ref{z_estimates} still holds and the
relation $H^{p+1}\simeq h^{q+1}$ gives an indication to choose the
relative step sizes $h$ and $H$ of ${\cal T}_H$ and ${\cal T}_h$ .
As above, the condition $h\textless \textless H$ is necessary to
expect a CPU time reduction.
\end{itemize}
\begin{remark}
We focus here on $P_k$ elements but the bi-grid method could apply to FEM spaces built with rectangular or cubic elements $Q_k$.
\end{remark}
%
\subsection{Two-grid schemes for Allen-Cahn Equation}
We now present two bi-grid schemes based on two different choices  of ${\tilde f}(u_h^k,{\tilde u}_h^{k+1},z_h^{k+1})$ in the case $f(u)=u(u^2-1)$:
\begin{itemize}
\item  [i.]${\tilde f}(u_h^k,{\tilde
u}_h^{k+1},z_h^{k+1})=f(u^k_h)$: the correction step of the
bi-grid scheme is a simple high mode stabilization of the
semi-implicit scheme. This choice defines the Scheme 4.1 presented
below. \item [ii.] Linearization for the nonlinear term:
$$
\begin{array}{ll}
{\tilde f}(u_h^k,{\tilde u}_h^{k+1},z_h^{k+1})&=
              \Frac{1}{4\epsilon^2}(u_h^{k}+2\tilde{u}_h^{k+1}u_h^{k}+3\tilde{u}_h^{k+1}-2)z_h^{k+1}\\
              &+
\Frac{1}{4\epsilon^2}
((\tilde{u}_h^{k+1})^2+(u_h^{k})^2-2)(\tilde{u}_h^{k+1}+u_h^{k}).
\end{array}
$$
This choice defines the Scheme 4.2 presented below.
\end{itemize}

\begin{minipage}[H]{16cm}
  \begin{algorithm}[H]
    \caption{Scheme 4.1: Two-grid Stabilized Allen-Cahn equation with correction}\label{StabACZ1}
    \begin{center}
    \begin{algorithmic}[1]
        \State $u_h^{0},u^{0}_H$ given\\
            \For{$k=0,1, \cdots$}
             \State {\bf Solve} $(u^{k+1}_H,\psi_H)+\Delta t(\nabla u_H^{k+1},\nabla \psi_H)=
             (u^{k}_H,\psi_H), \; \forall \psi_H\in
                W_H$
                \State {} \hskip 0.5cm $+\Delta t\Frac{1}{\epsilon^2}(DF(u_H^{k+1},u_H^{k}),\psi_H)$\label{SchemaRef}
              \State {\bf Solve} $({\tilde u}^{k+1}_h-u_H^{k+1},\phi_h)=0,\; \forall \phi_h\in V_h$
               \State {\bf Solve} $(1+\tau \Delta t)(z^{k+1}_h,\phi_h)
               +\Delta t(\nabla z^{k+1}_h,\nabla \phi_h)
               =(1+\tau \Delta t)(z^{k}_h,\phi_h)$
              \State {}\hskip 5.cm $- \Delta t(\nabla {\tilde u}_h^{k+1},\nabla \phi_h)
             -({\tilde u}^{k+1}_h-{\tilde u}^{k}_h,\phi_h)$
              \State {}\hskip 5.cm $-\Frac{\Delta t}{4\epsilon^2}(u_h^{k}((u^{k})^2-1),\phi_h) \forall \phi_h\in V_h$
              \State {\bf Set}  $u_h^{k+1}={\tilde u}_h^{k+1}+z_h^{k+1}$
            \EndFor
    \end{algorithmic}
    \end{center}
    \end{algorithm}
\end{minipage}
\medskip
\begin{remark}
It is important to note that the above scheme can be implemented very simply without computing explicitly the sequence $z^k_h$: indeed, we can rewrite the correction step as
$$
\begin{array}{ll}
(u^{k+1}_h,\phi_h)& \\
+\Delta t \tau (u^{k+1}_h-u^{k}_h,\phi_h)&  =(u^{k}_h,\phi_h)+\Delta t \tau (u^{k+1}_H-u^{k}_H,\phi_h)-\Delta t\Frac{1}{\epsilon^2}(f(u_h^{k}),\phi_h), \; \forall \phi_h\in V_h.\\
+\Delta t(\nabla u_h^{k+1},\nabla \phi_h) &
\end{array}
$$
\end{remark}
\medskip
\begin{minipage}[H]{16cm}
  \begin{algorithm}[H]
    \caption{Scheme 4.2: Two-grid Stabilized Allen-Cahn equation with correction}\label{StabACZ2}
        \begin{center}
    \begin{algorithmic}[1]
        \State $u_h^{0},u^{0}_H$ given\\
            \For{$k=0,1, \cdots$}
             \State {\bf Solve} $(u^{k+1}_H,\psi_H)+\Delta t(\nabla u_H^{k+1},\nabla \psi_H)=
             (u^{k}_H,\psi_H), \; \forall \psi_H\in
                W_H$
                \State {} \hskip 0.5cm $+\Delta t(\Frac{1}{\epsilon^2}DF(u_H^{k+1},u_H^{k}),\psi_H)$\label{SchemaRef}
              \State {\bf Solve} $({\tilde u}^{k+1}_h-u_H^{k+1},\phi_h)=0,\; \forall \phi_h\in V_h$
               \State {\bf Solve} $(1+\tau \Delta t)(z^{k+1}_h,\phi_h)
               +\Delta t(\nabla z^{k+1}_h,\nabla \phi_h)
               =(1+\tau \Delta t)(z^{k}_h,\phi_h)$
              \State {}\hskip 5.cm $- \Delta t(\nabla {\tilde u}_h^{k+1},\nabla \phi_h)
             -({\tilde u}^{k+1}_h-{\tilde u}^{k}_h,\phi_h)$
              \State {}\hskip 5.cm $-\ds\Frac{\Delta
              t}{4\epsilon^2}((u_h^{k}+2\tilde{u}_h^{k+1}u_h^{k}+3\tilde{u}_h^{k+1}-2)z_h^{k+1},\phi_h)$
\State {}\hskip 5.cm $-\ds\frac{\Delta t}{4\epsilon^2}
(((\tilde{u}_h^{k+1})^2+(u_h^{k})^2-2)(\tilde{u}_h^{k+1}+u_h^{k}),\phi_h),
\forall \phi_h\in V_h$
              \State {\bf Set}  $u_h^{k+1}={\tilde u}_h^{k+1}+z_h^{k+1}$
            \EndFor
    \end{algorithmic}
     \end{center}
    \end{algorithm}
\end{minipage}
\medskip
\begin{remark}
The stabilization we use here applies on the high modes components,
this can be compared to the methods developed by
Costa-Dettori-Gottlieb and Temam \cite{CDGT} when using spectral
methods (Fourier, Chebyshev) or Chehab-Costa
\cite{ChehabCosta1,ChehabCosta2} in finite differences: in these
cases several grids were used for generating a hierarchy of
fluctuant component in embedded grids and to stabilize them with
as damping term as above; however the approach we propose here can be
non hierarchical and can be applied for many choices of FEM
spaces. In a recent work, one grid stabilization was proposed in
finite difference for parabolic equations using preconditioning
techniques, \cite{BrachetChehab}.
\end{remark}
%
\subsection{Global stabilization vs high mode stabilization}
Before comparing the performances of the bi-grid method and the
one-grid reference scheme, we would like to illustrate the effect
of the high mode stabilization with respect to the global
stabilization, in the time evolution of the  energy.
\subsubsection{High mode stabilization of the semi-implicit scheme (Scheme 4.1)}
Here $\Omega=]0,1[^2$ and two triangulations ${\cal T}_h$ and ${\cal T}_H$ are considered; they are composed of 1681 and 441 triangles respectively. Both
$W_H$ and $V_h$ are FEM spaces built on $P_2$ elements, their dimensions are $\mbox{dim}(W_H)=1681$ and $\mbox{dim}(V_h)=6561$, so
$DR(W_H,v_h)=0.256211$.\\\\
The stabilization applied to the only high mode components allows
to compute the solution with a good accuracy; the energy history
of the one-grid reference scheme is close to the one of the
bi-grid one while the stabilization of the scheme on all the
components of the solution slows down the dynamics. The
stabilization term is of course necessary for the one-grid scheme
but also for the two grid scheme: on the same example as above,
for $S=0.05$ both schemes scheme 4.2 and scheme 2 are unstable.
\begin{figure}[h!]
\begin{center}
\vspace{-1.3cm}
\includegraphics[width=6.cm, height=10.5cm]{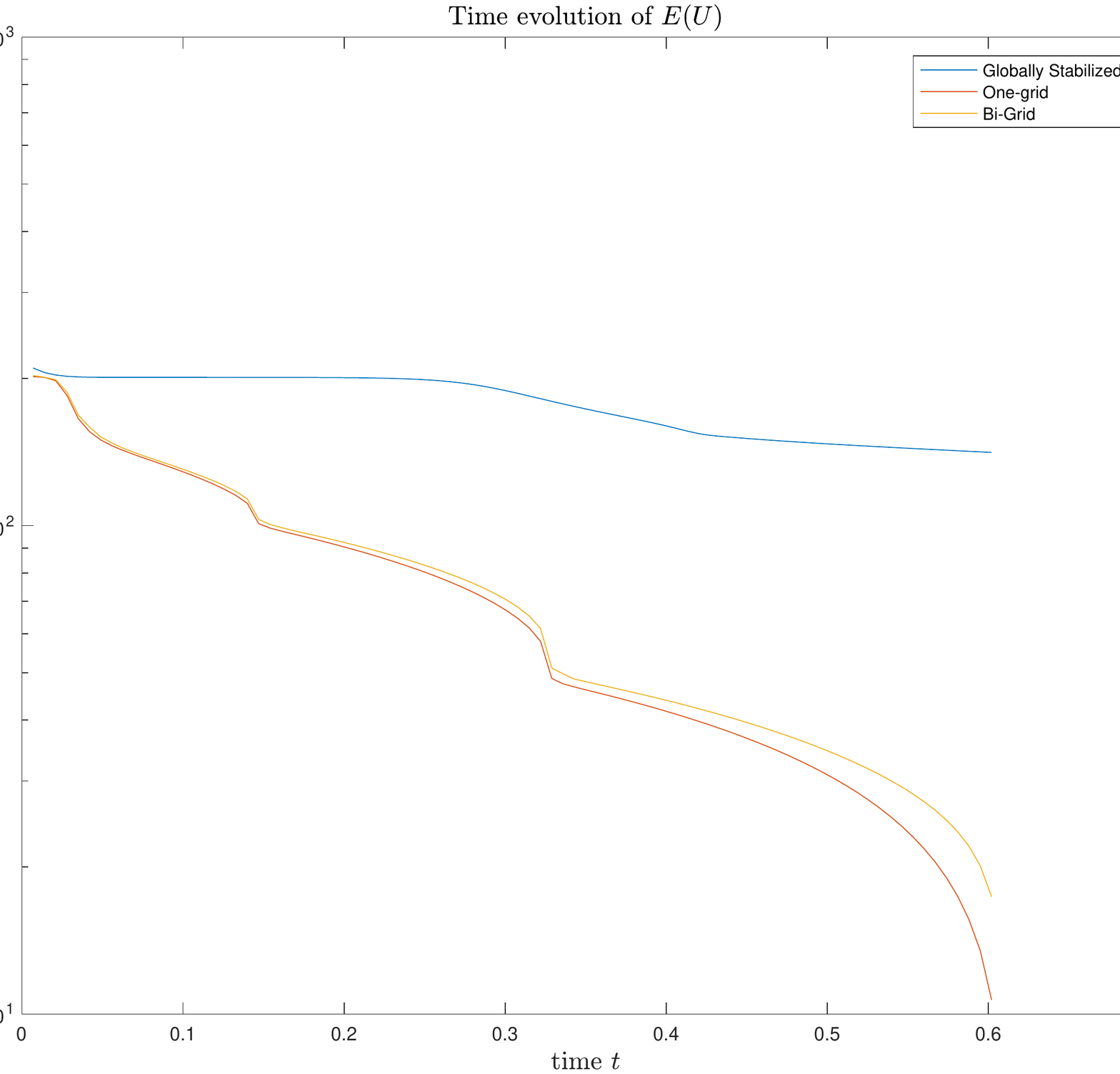}\hskip 1.5cm
\end{center}
\vspace{-2.3cm} \caption{Allen-Cahn equation - Energy vs time -
Comparison between globally stabilized one grid method, one grid
method and high mode stabilized bi-grid method, $\epsilon=0.03$,
$\Delta t=7\times10^{-3}$, $\tau=S/\epsilon^2$, $S=2$ (left),
$S=2$ (right), $u_0(x,y)=\cos(4 \pi x) \cos(4\pi y)$
} \label{Comp_Shen2}
\end{figure}
\subsubsection{High mode stabilization via a proper linearization for the nonlinear term (Scheme 4.2)}
\begin{figure}[!hp]
\vspace{-1.9cm}
\begin{center}
\includegraphics[width=6.cm, height=10cm]{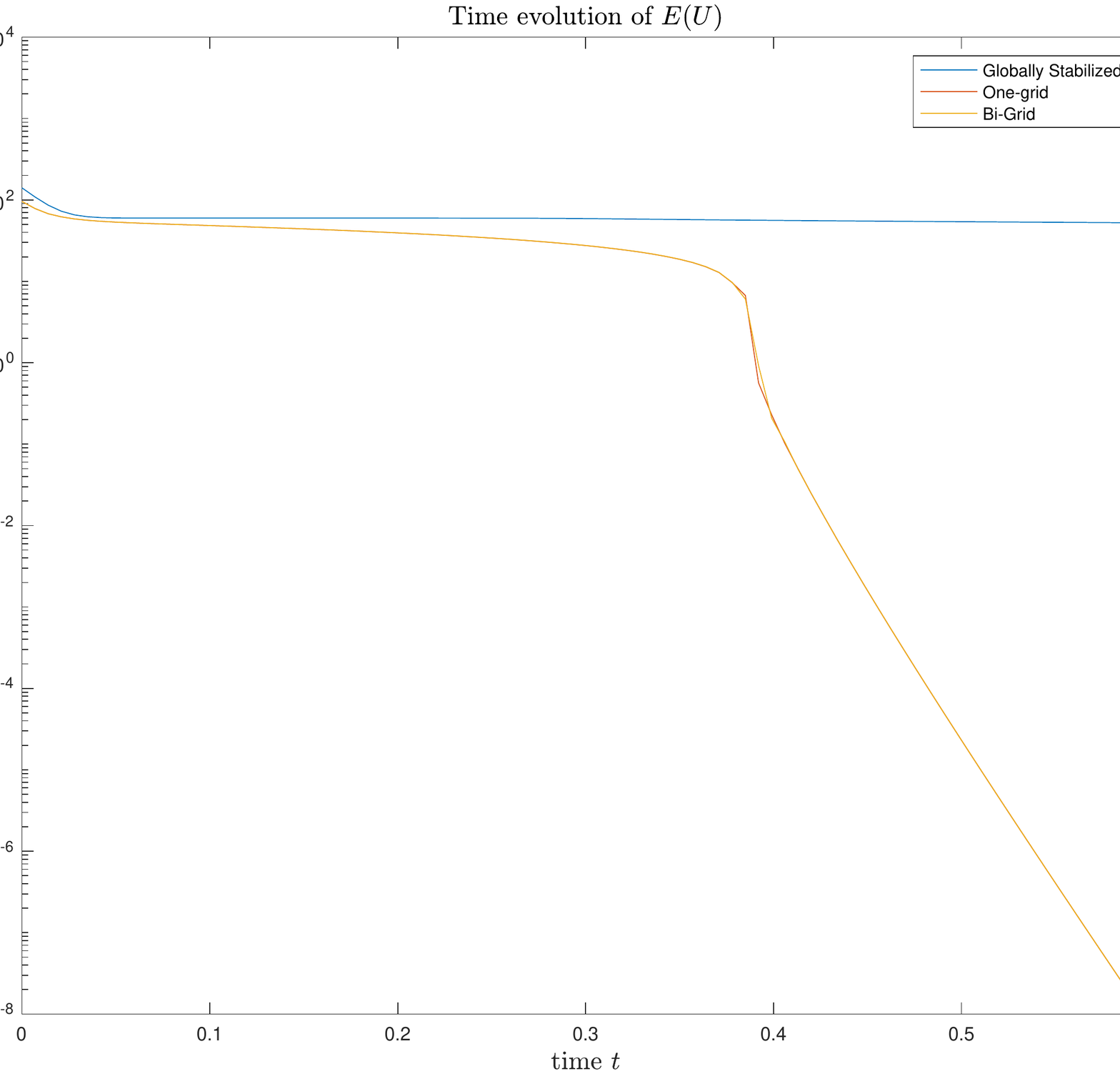}\hskip 1.5cm
\includegraphics[width=6.cm, height=10cm]{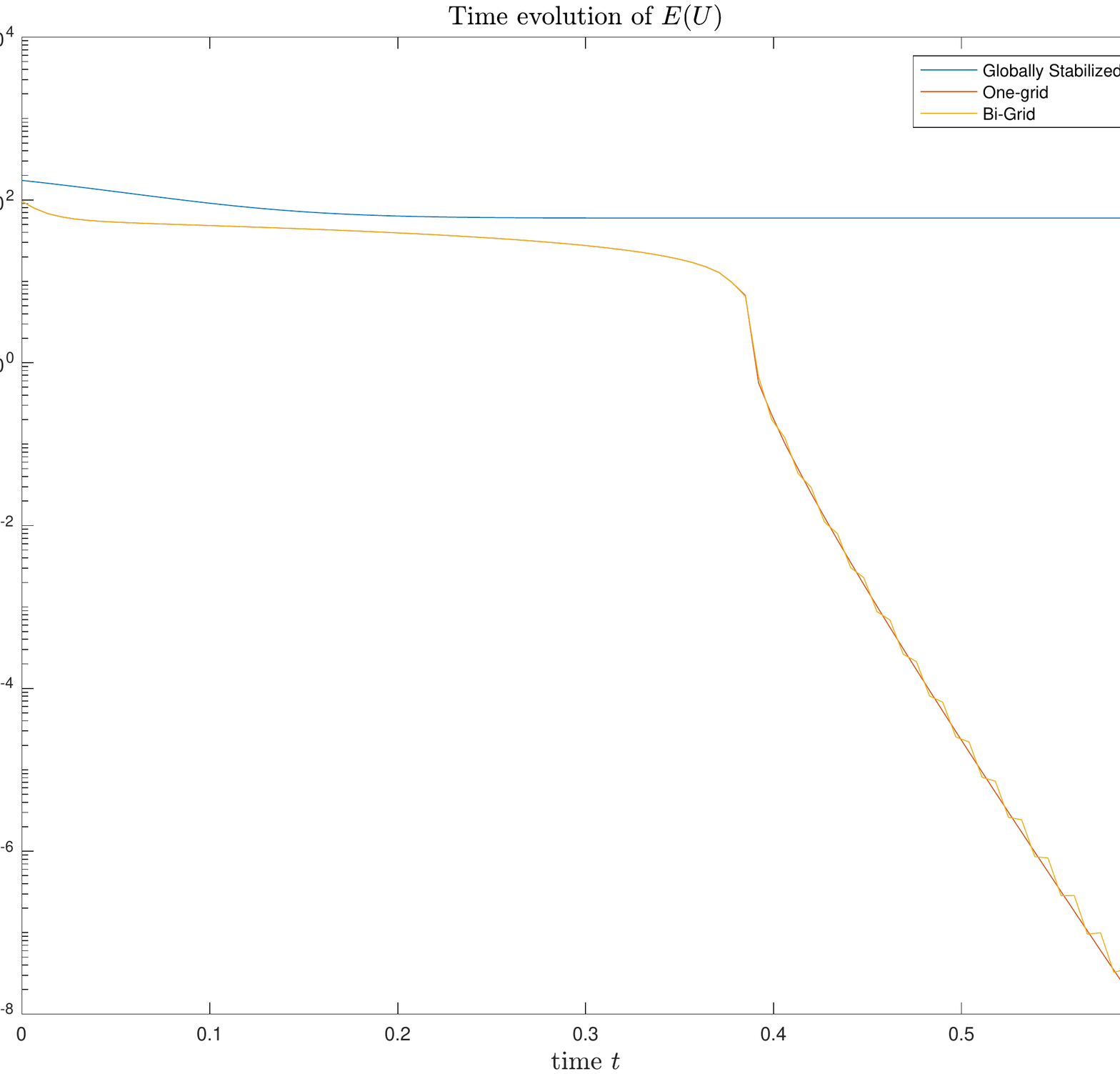}\\
\vspace{-2.3cm}
\end{center}
\caption{Allen-Cahn equation - Energy vs time - Comparison between
globally stabilized one grid method, one grid method and bi-grid
method. $\epsilon=0.03$, $\Delta t=7\times10^{-3}$,
$\tau=S/\epsilon^2$, $S=1.5$ (left), $S=10$ (right),
$u_0(x,y)=\cos(\pi x) \cos(\pi y)$} \label{Comp_Shen1}
\end{figure}
%
%
\clearpage
\begin{figure}[h!]
\vspace{-1.5cm}
\begin{center}
\includegraphics[width=6.cm, height=10cm]{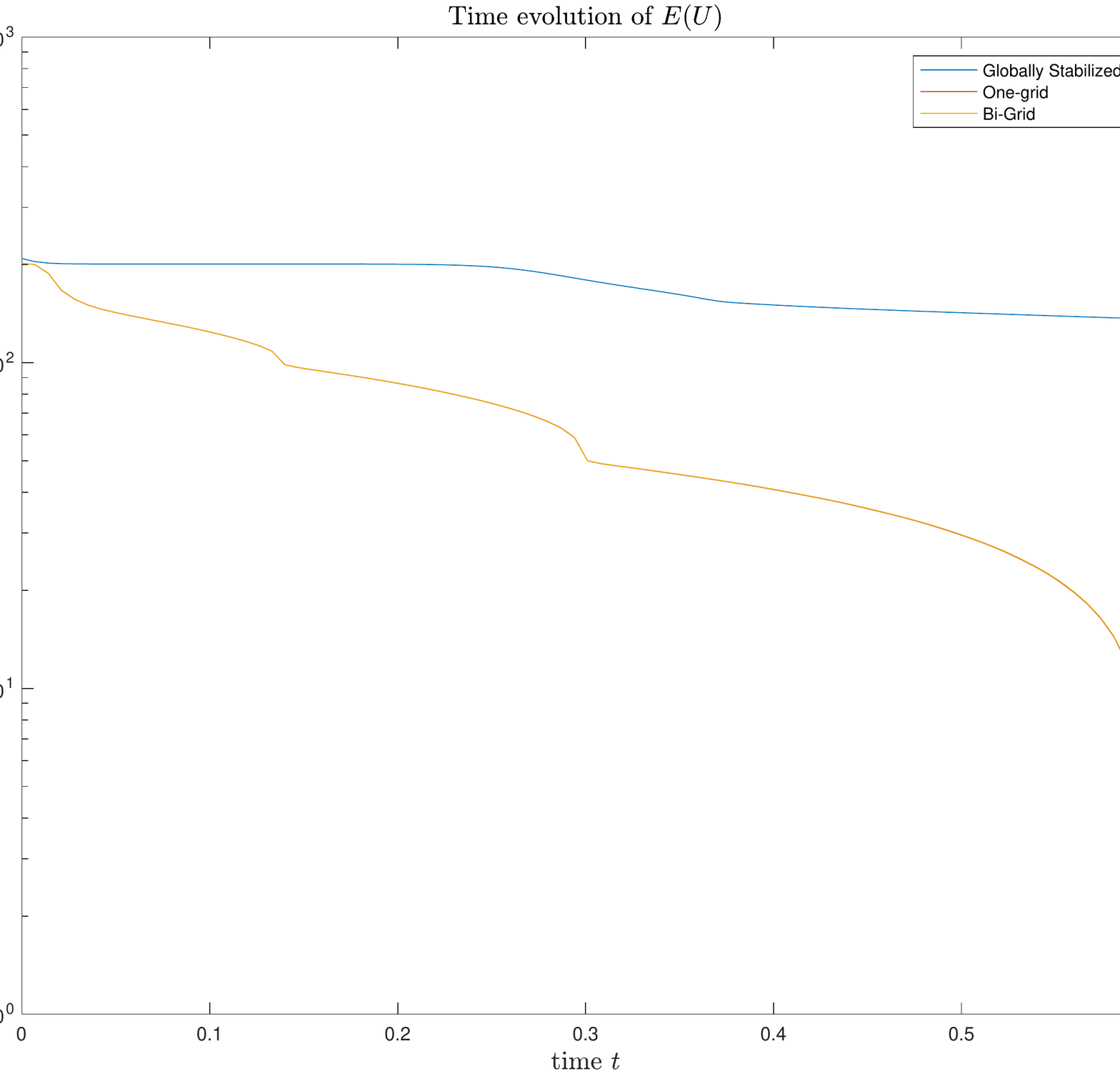}\hskip 1.5cm
\includegraphics[width=6.cm, height=10cm]{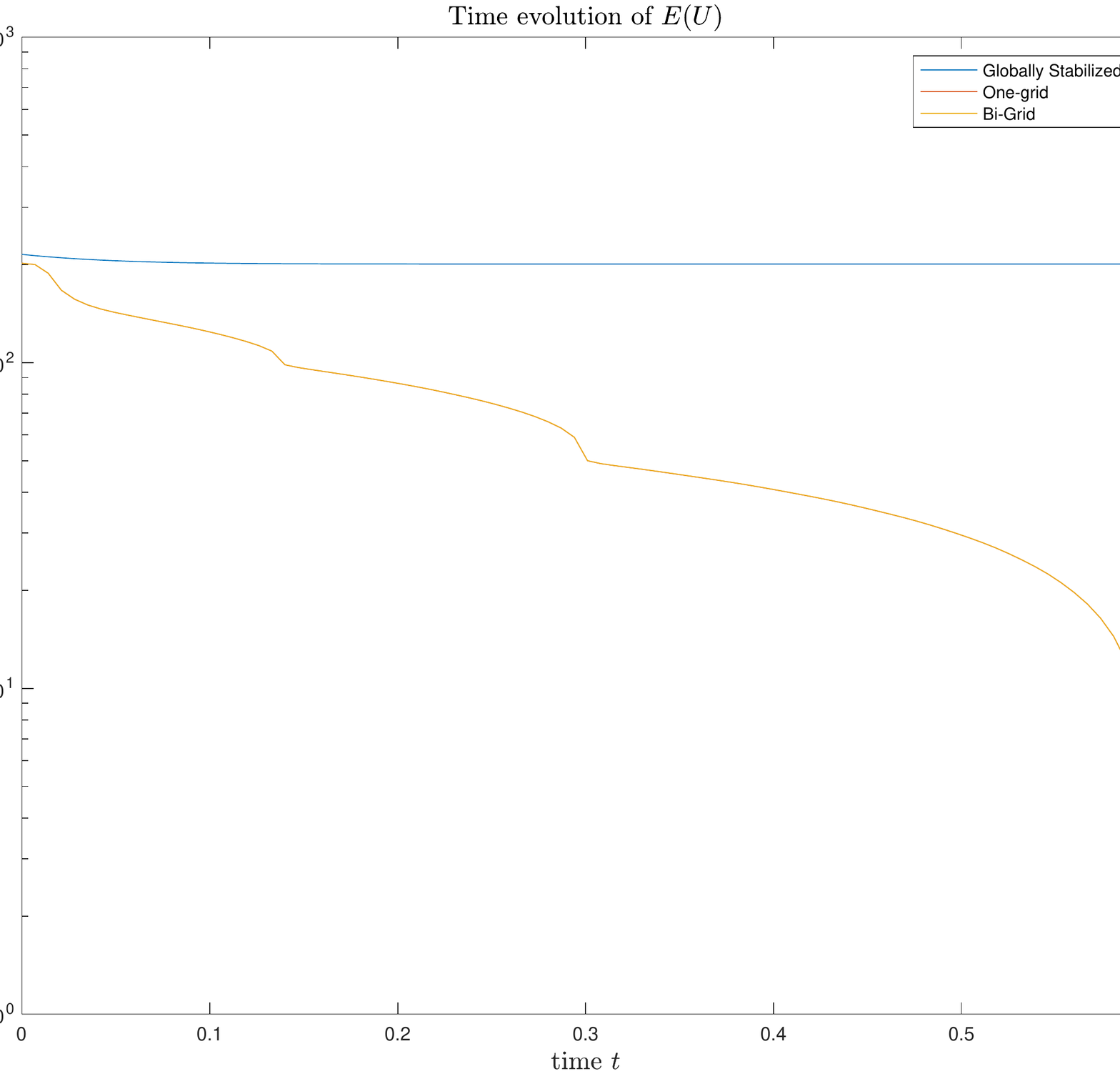}\\
\vspace{-0.2cm}
\end{center}
\vskip -2.cm
\caption{Allen-Cahn equation - Energy vs time - Comparison between Globally stabilized one grid method, one grid method and bi-grid method.
$\epsilon=0.03$, $\Delta
t=7\times10^{-3}$.
$\tau=S/\epsilon^2$, $S=1.5$ (left), $S=10$ (right). $u_0(x,y)=\cos(4 \pi x) \cos(4\pi y)$}
\label{Comp_Shen2}
\end{figure}

We have plotted for different initial data in Figures
\ref{Comp_Shen1} and \ref{Comp_Shen2} the time evolution of the
energy for the 3 methods. We observe  that for same stabilization
parameters, the bi-grid schemes based on high mode damping
restitue an energy dynamics comparable to that of the reference
scheme  while the one-grid stabilization slows down the decreasing
of the energy.  Also, as expected, higher values of $\tau$
produces more important energy slow down for the stabilized
one-grid scheme. Same results are obtained using $P_1$ elements
instead of $P_2$.
\begin{table}[!hp]
\begin{center}
\begin{tabular}{|l|l|l|l|l|l|}
\hline
Scheme & $\epsilon$ & $S$ & $\tau=S/\epsilon^2$ & $\Delta t$ & Stability\\
\hline
\hline
one-grid stab Scheme 2& $0.03$ & $0.1$ & $111.11$& $0.007$&  yes\\
\hline
bigrid Scheme 3 & $0.03$ & $0.1$ &$111.11$ & $0.007$&  yes\\
\hline
\hline
one-grid stab Scheme 2& $0.03$ & $0.05$ &$55.55$ & $0.007$&  no\\
\hline
bigrid Scheme 3 & $0.03$ & $0.05$ & $55.55$& $0.007$&  yes\\
\hline
\hline
one-grid stab Scheme 2& $0.03$ & $0.01$ & $1.11$& $0.007$&  no\\
\hline
bigrid Scheme 3 & $0.03$ & $0.01$ & $11.11$& $0.007$&  yes\\
\hline
\end{tabular}
\end{center}
\caption{$\epsilon=0.03, \ \Delta t=0.007, \ u_0(x,y)=\cos(4 \pi x) \cos(4\pi y)$, $P_2$ elements}
\end{table}
%
%
\vspace{-0.6cm}
\begin{table}[!hp]
\begin{center}
\begin{tabular}{|l|l|l|l|l|l|}
\hline
Scheme & $\epsilon$ & $S$ & $\tau=S/\epsilon^2$ & $\Delta t$ & Stability\\
\hline
\hline
one-grid stab Scheme 2& $0.03$ & $10$ &$111111.11$ & $0.01$&  yes\\
\hline
bigrid Scheme 3 & $0.03$ & $10$ & $111111.11$& $0.01$&  yes\\
\hline
\end{tabular}
\end{center}
\caption{$\epsilon=0.03, \ \Delta t=0.01, \ u_0(x,y)=\cos(4 \pi x) \cos(4\pi y)$, $P_1$ elements}
\end{table}
\clearpage
\begin{table}[!hp]
\begin{center}
\begin{tabular}{|l|l|l|l|l|l|}
\hline
Scheme & $\epsilon$ & $S$ & $\tau=S/\epsilon^2$ & $\Delta t$ & Stability\\
\hline
\hline
one-grid stab Scheme 2& $0.03$ & $0.1$ & $111.11$& $0.007$&  no\\
\hline
bigrid Scheme 3 & $0.03$ & $0.1$ &$111.11$ & $0.007$&  no\\
\hline
\end{tabular}
\end{center}
\caption{$\epsilon=0.03, \ \Delta t=0.01, \ u_0(x,y)=\cos(8 \pi x) \cos(7\pi y)$, $P_2$ elements}
\end{table}
We finish with two illustrations  obtained in the case in which
the same triangulation is used for the two FEM spaces $W_H$ and $V_h$ is
used. We choose
$$
W_H=\{ v_H\in {\cal C}^0({\bar \Omega}) / v_H|_{K}\in P_1, \
\forall K\in {\cal T}_h\} \mbox{ and } V_h=\{ v_h\in {\cal
C}^0({\bar \Omega}) / v_h|_{K}\in P_2, \ \forall K\in {\cal T}_h\}.
$$
We have $W_H\subset V_h$. We first consider the Allen-Cahn equation on the torus $\Omega=\{(x,y)\in R^2 / 1\le x^2+y^2\le 9\}$.
The dimensions are $\mbox{dim}(V_h)=1606$,  $\mbox{dim}(W_H)=424$, $DR(W_H,V_h)=0.26401$.\\
The time evolution of the energy for the three methods is reported
on Figure \ref{Comp_Shen_torus} (left); the conclusions are the
same as in the previous examples: the high mode stabilization
allows to stabilize without deteriorating the dynamics. We also
consider $\Omega=]0,1[^2$ and $\mbox{dim}(V_h)=14641$,
$\mbox{dim}(W_H)=3721$, $DR(W_H,V_h)=0.254149$. The time evolution
of the energy is related on Figure (\ref{Comp_Shen_torus})
(right).
\begin{figure}[h!]
\begin{center}
\vspace{-0.8cm}
\includegraphics[width=6.cm, height=10cm]{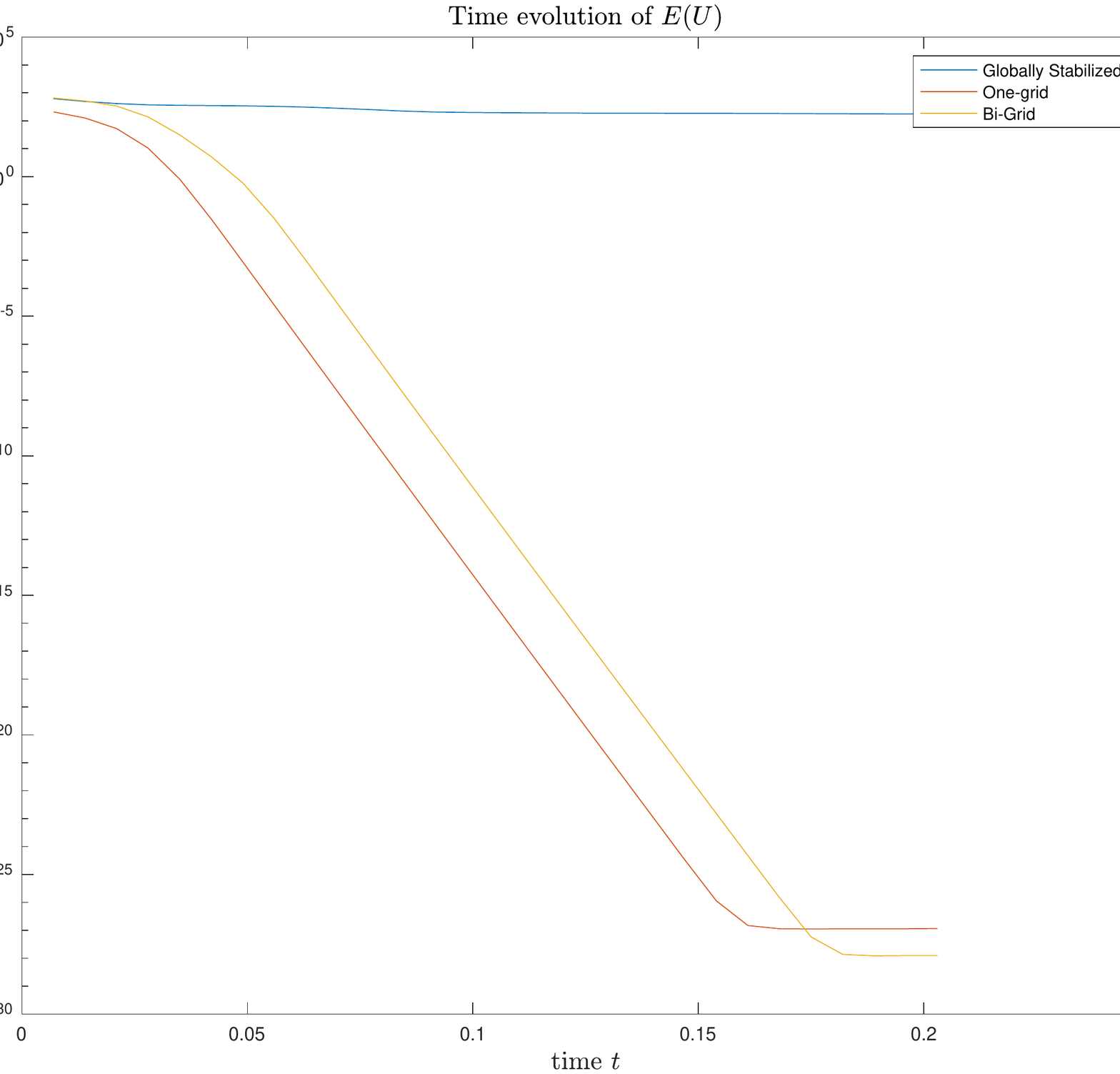}\hskip 1.5cm
\includegraphics[width=6.cm, height=10cm]{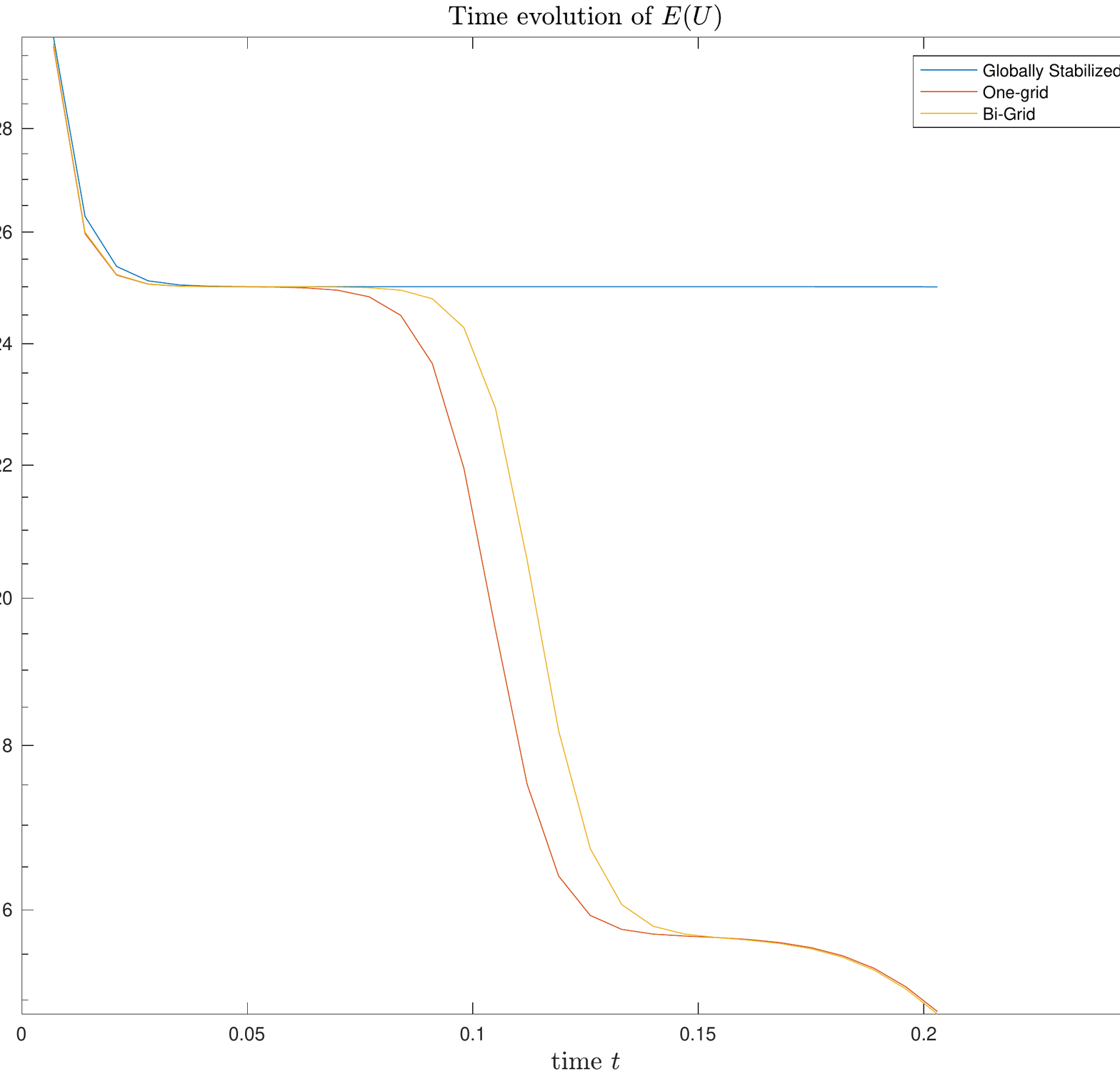}
\end{center}
\vspace{ -2.cm} \caption{(left) $\Omega$ is  the torus
(R=1, R=3)- Energy vs time - Comparison between Globally
stabilized one grid method, one grid method and high mode
stabilized bi-grid method with linearization. $\epsilon=0.1$,
$\Delta t=7\times10^{-3}$. $\tau=S/\epsilon^2$, $S=0.01$ (left),
$S=2$ (right). $u_0(x,y)=cos((x^2+y^2-1)*(x^2+y^2-9))^2$.
(right) $\Omega$ is the unit square- Energy vs time -
Comparison between Globally stabilized one grid method, one grid
method and high mode stabilized bi-grid method with linearization.
$\epsilon=0.1, \Delta t=7\times10^{-3}, \tau=S/\epsilon^2, S=0.1,
u_0(x,y)=\cos(4\pi x)\cos(4\pi y)$, $P_1$ elements on  $W_H$, $P_2$ on $V_h$}
\label{Comp_Shen_torus}
\end{figure}
\clearpage
\subsubsection{Direct Simulation}
Here $\Omega=]0,1[^2$ and two triangulations ${\cal T}_h$ and ${\cal T}_H$ are considered; they are composed of 1681 and 441 triangles respectively.
The dimensions of the FEM spaces are $\mbox{dim}(W_H)=1681$ and $\mbox{dim}(V_h)=6561$, so $DR(W_H,V_h)=0.256211$.\\
To illustrate the robustness of the bi-grid method (Scheme 4.2),
we make comparison with the one grid unconditionally stable
scheme (Scheme 2). We observe that time history of the energy is
comparable and that the time-evolution of the difference between
the solutions produced by these schemes remains small. Finally, we
observe that the CPU time is reduced for the bi-grid scheme, by a
factor larger than 4. \begin{figure}[h!]
\begin{center}
\vspace{-3.2cm}
\includegraphics[width=7.8cm, height=10.5cm]{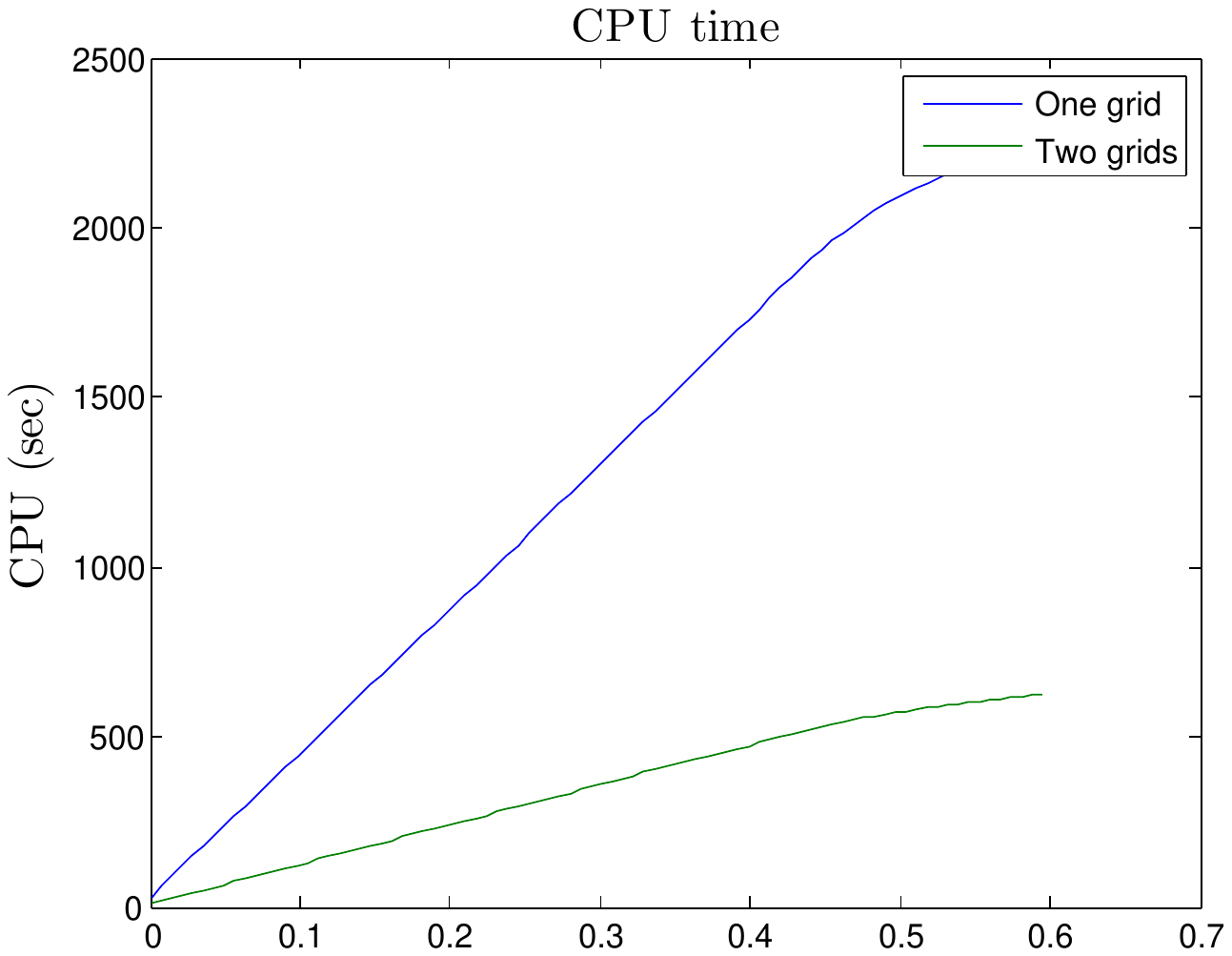}
\includegraphics[width=7.8cm, height=10.5cm]{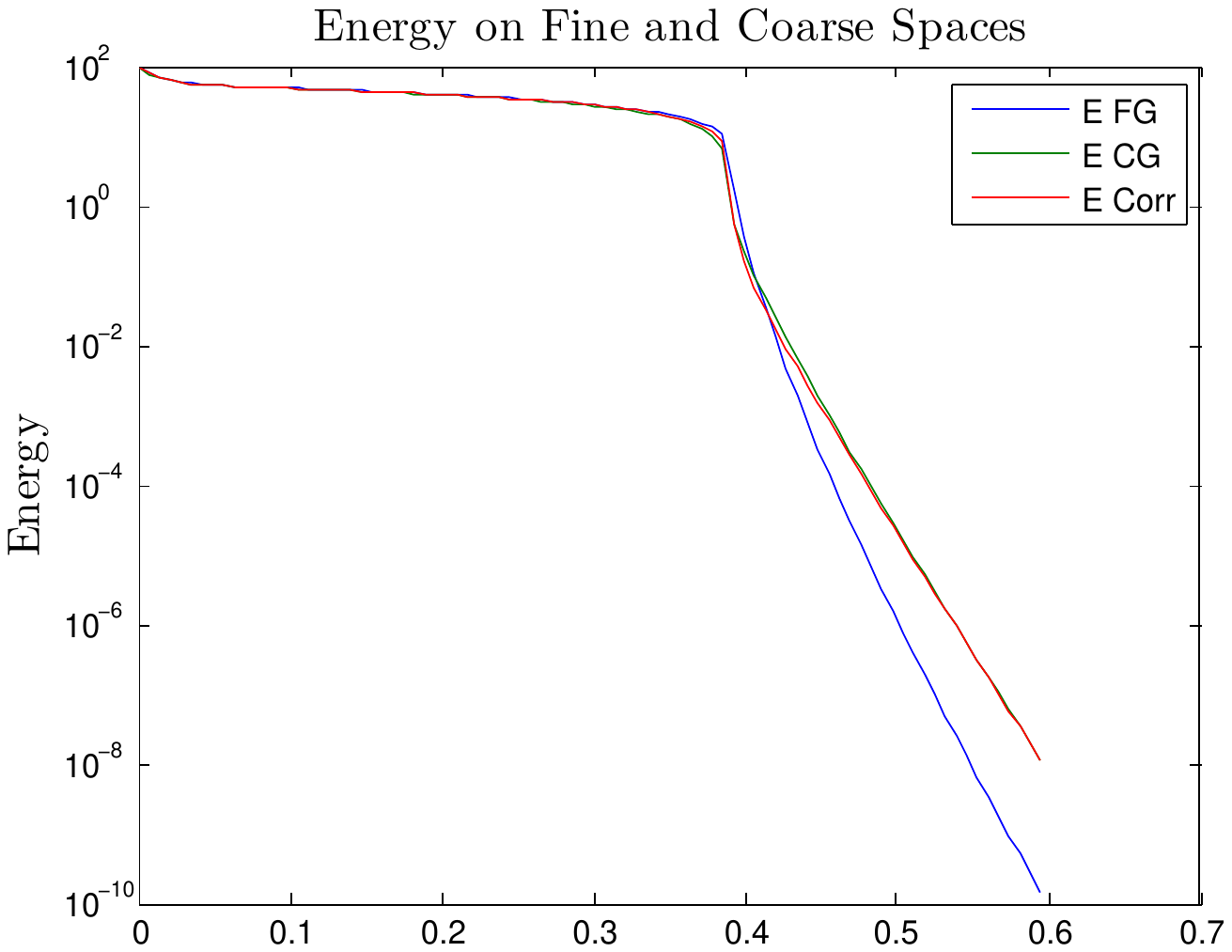}\\
\vskip -4.5 cm
\includegraphics[width=9.5cm,height=8cm]{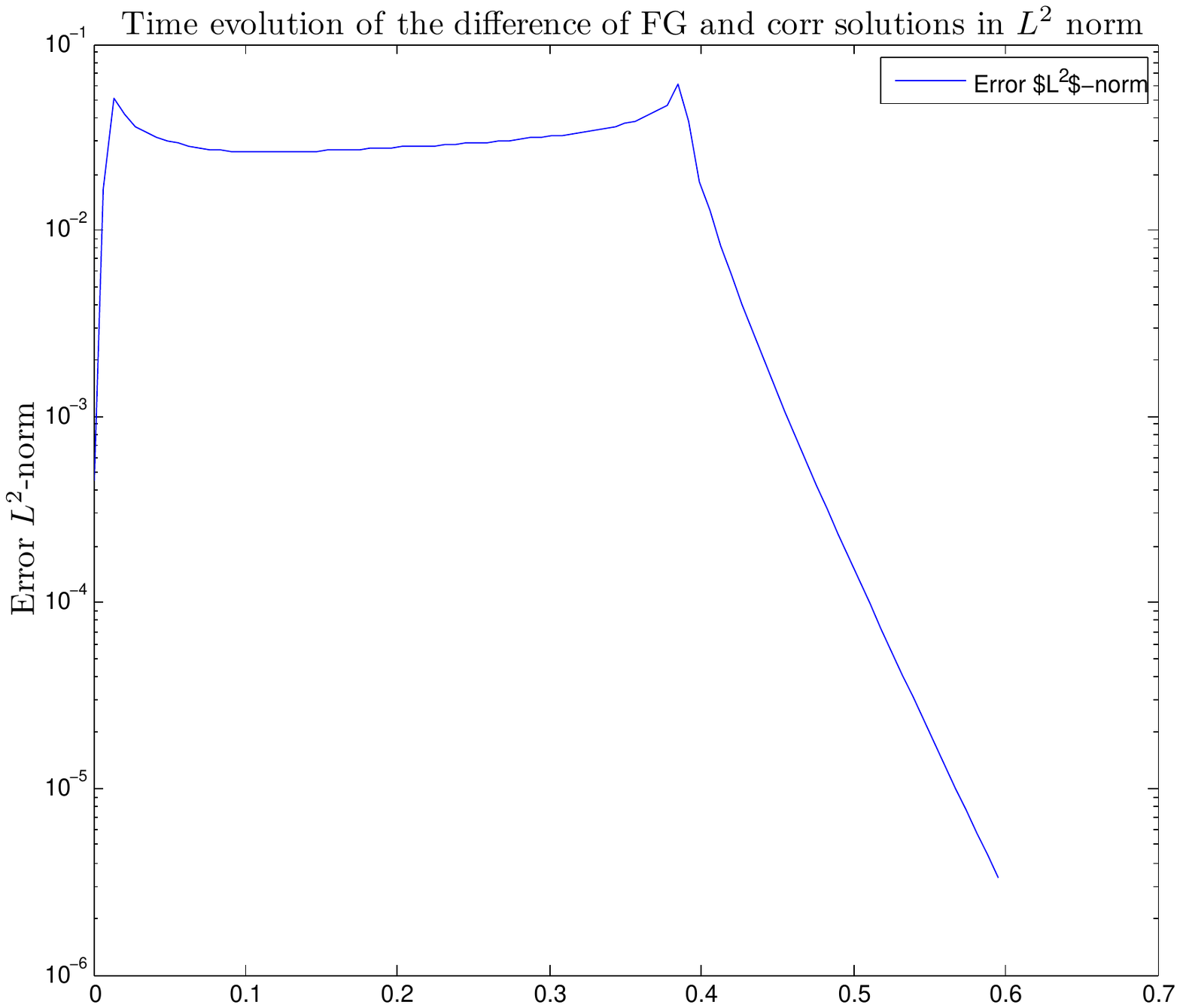}
\vspace{-0.2cm}
\end{center}
\vskip -2.cm
\caption{Allen-Cahn equation
$\epsilon=0.03,\tau=5\times 10^{-3}$ and $\Delta
t=7\times10^{-3}$.}
\end{figure}
\clearpage
\subsubsection{Simulation of an exact solution}
We here always  consider $\Omega=]0,1[^2$ two triangulations
${\cal T}_h$ and ${\cal T}_H$ composed of 1681 and 441 triangles
respectively. The dimensions of the FEM spaces are
$\mbox{dim}(W_H)=1681$ and $\mbox{dim}(V_h)=6561$, so
$DR(W_H,V_h)=0.256211$. We now compare the bi-grid method (Scheme
4.2) and the one-grid Scheme 2 simulating an exact solution: we
choose the function $u_{ex}(x,y,t)=\cos(\pi x)\cos(\pi
y)exp(\sin(\pi t))$ to be the solution, an appropriate  r.h.s. is
added to this aim. As above, we observe that the error is small
while the CPU time is reduced by a factor 3  in Figure (\ref{AC_Exact1}), for $\epsilon=0.5$),  and larger than 4 in Figure (\ref{AC_Exact2}), for $\epsilon=0.3$.
\begin{figure}[h!]
\begin{center}
\includegraphics[width=5.8cm, height=5.5cm]{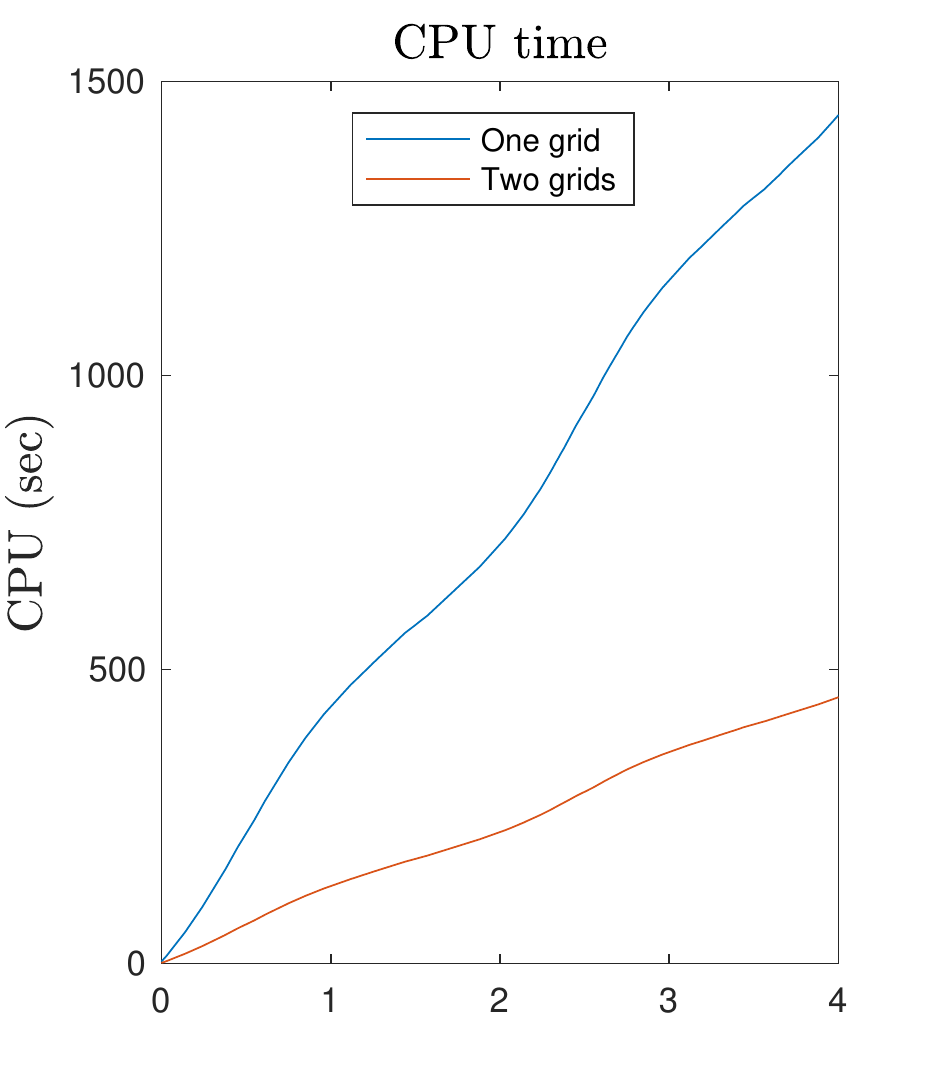}
\includegraphics[width=5.8cm, height=5.5cm]{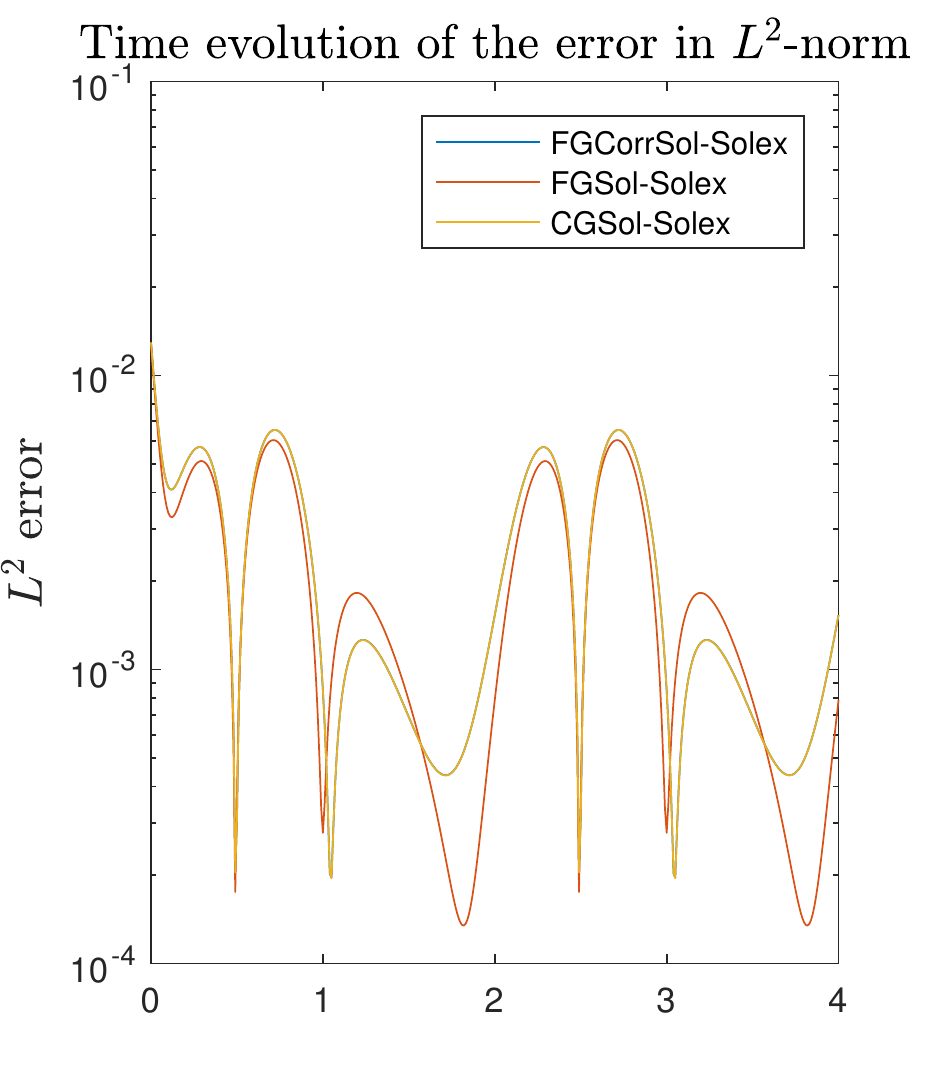}\\
\includegraphics[width=5.8cm, height=5.5cm]{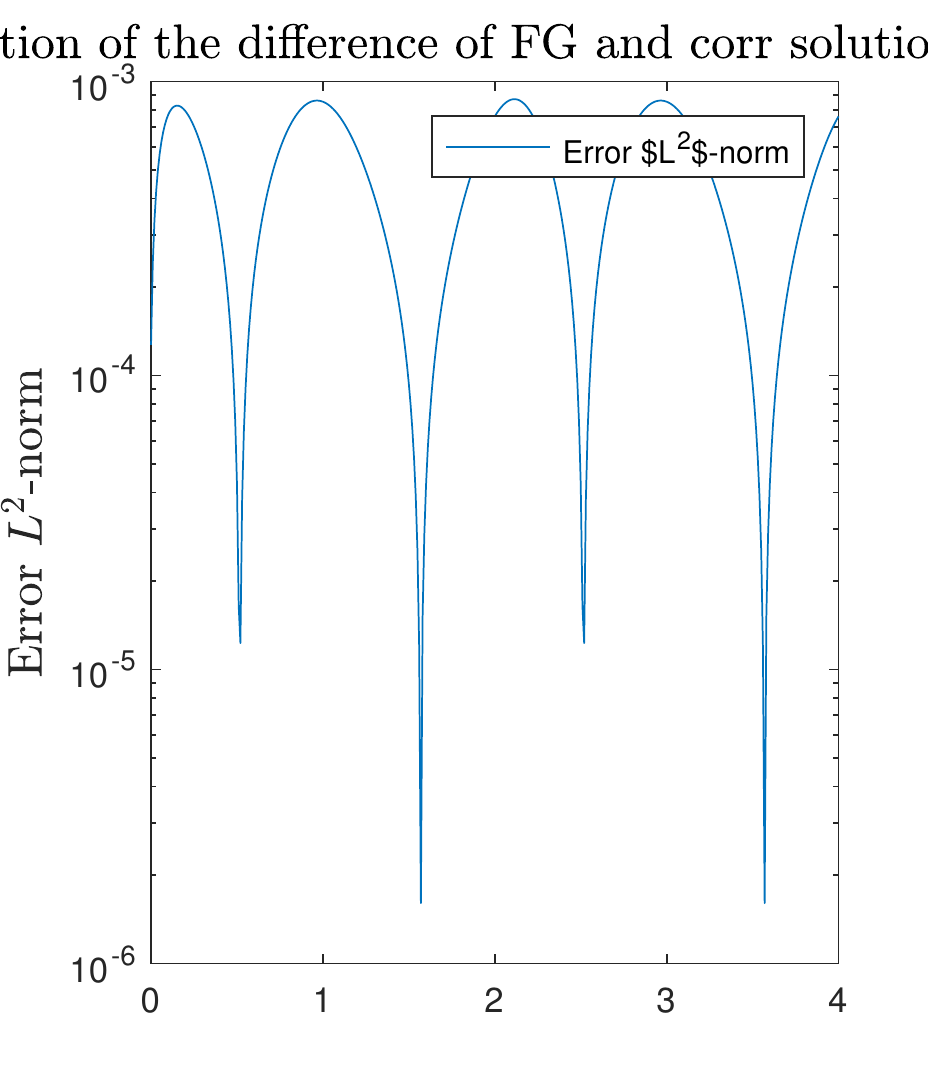}
\includegraphics[width=5.8cm, height=5.5cm]{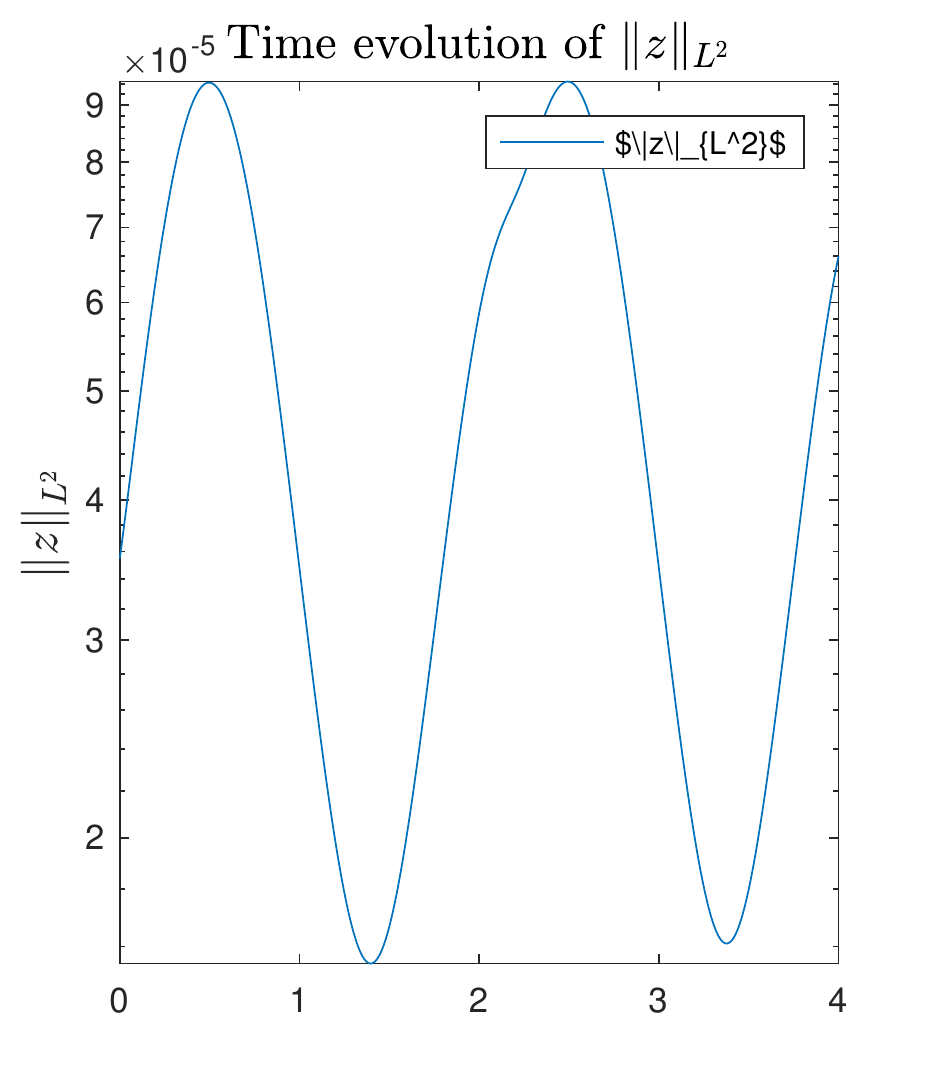}\\
\end{center}
\caption{Allen-Cahn - simulation of an exact solution
$\epsilon=0.5,\tau=12$, $\Delta t=10^{-2}$.}
\label{AC_Exact1}
\end{figure}
\begin{figure}[h!]
\begin{center}
\vspace{-1.9cm}
\includegraphics[width=7.8cm, height=10.5cm]{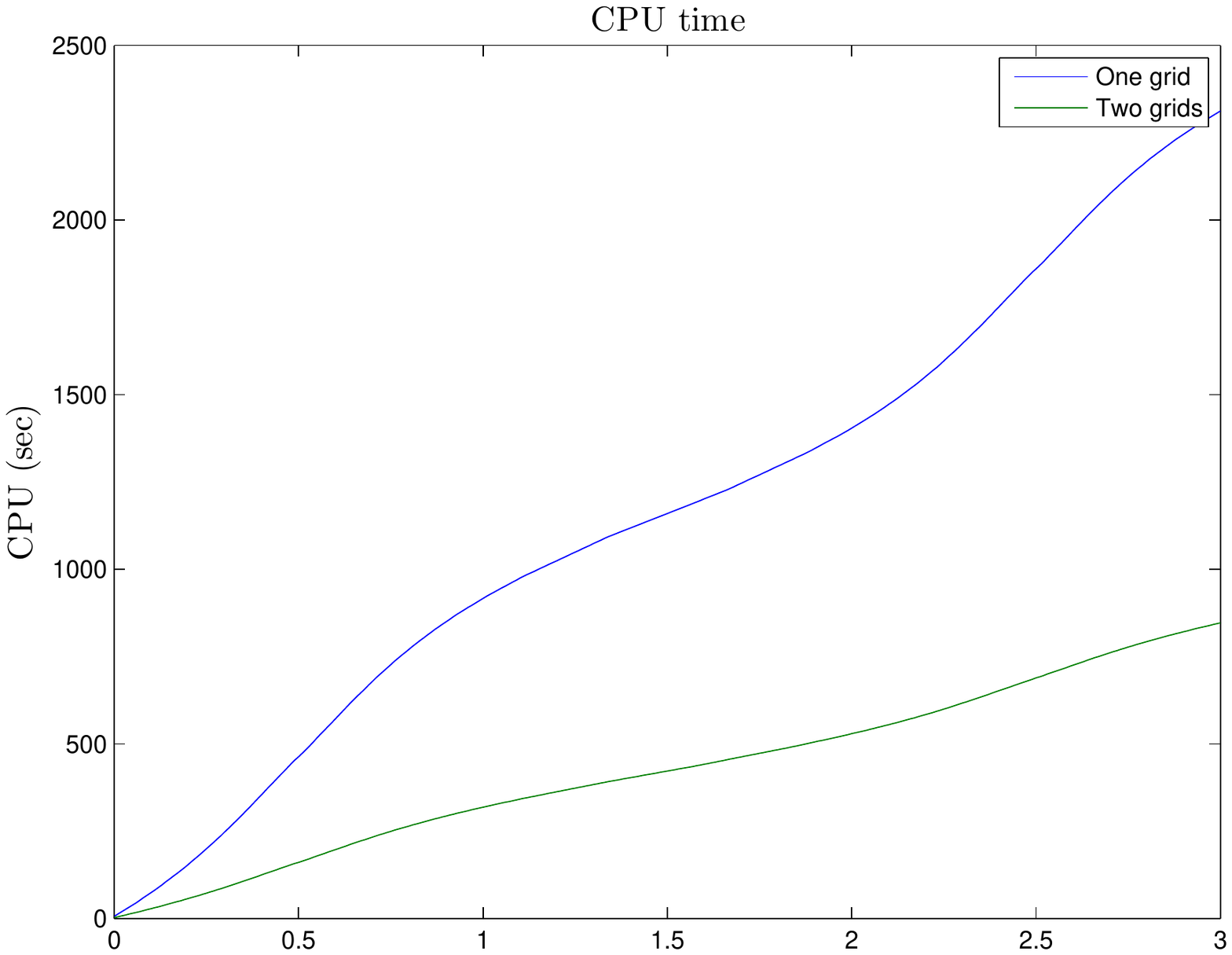}
\includegraphics[width=7.8cm, height=10.5cm]{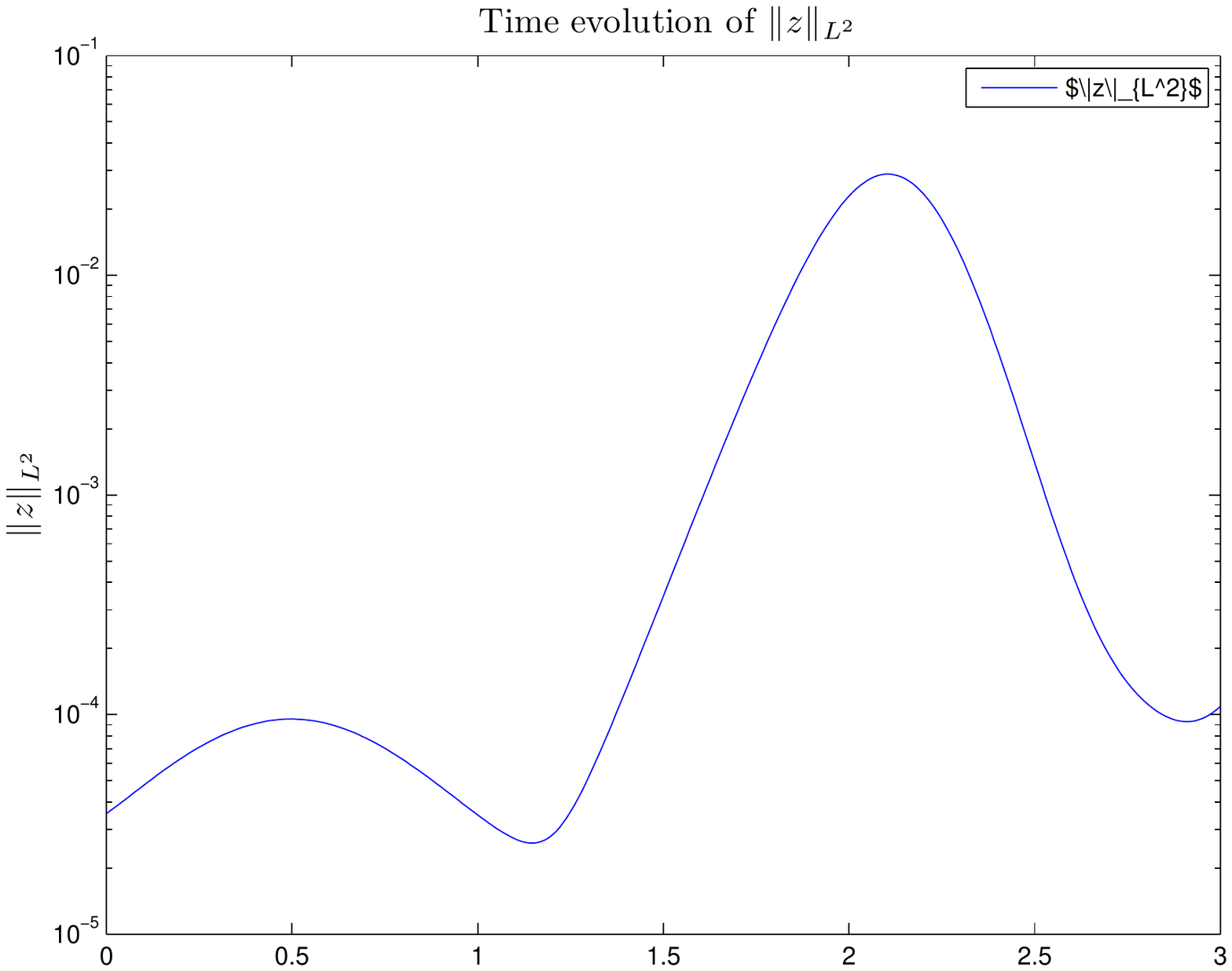}\\
\vspace{-4.9cm}
\includegraphics[width=9.8cm, height=10.5cm]{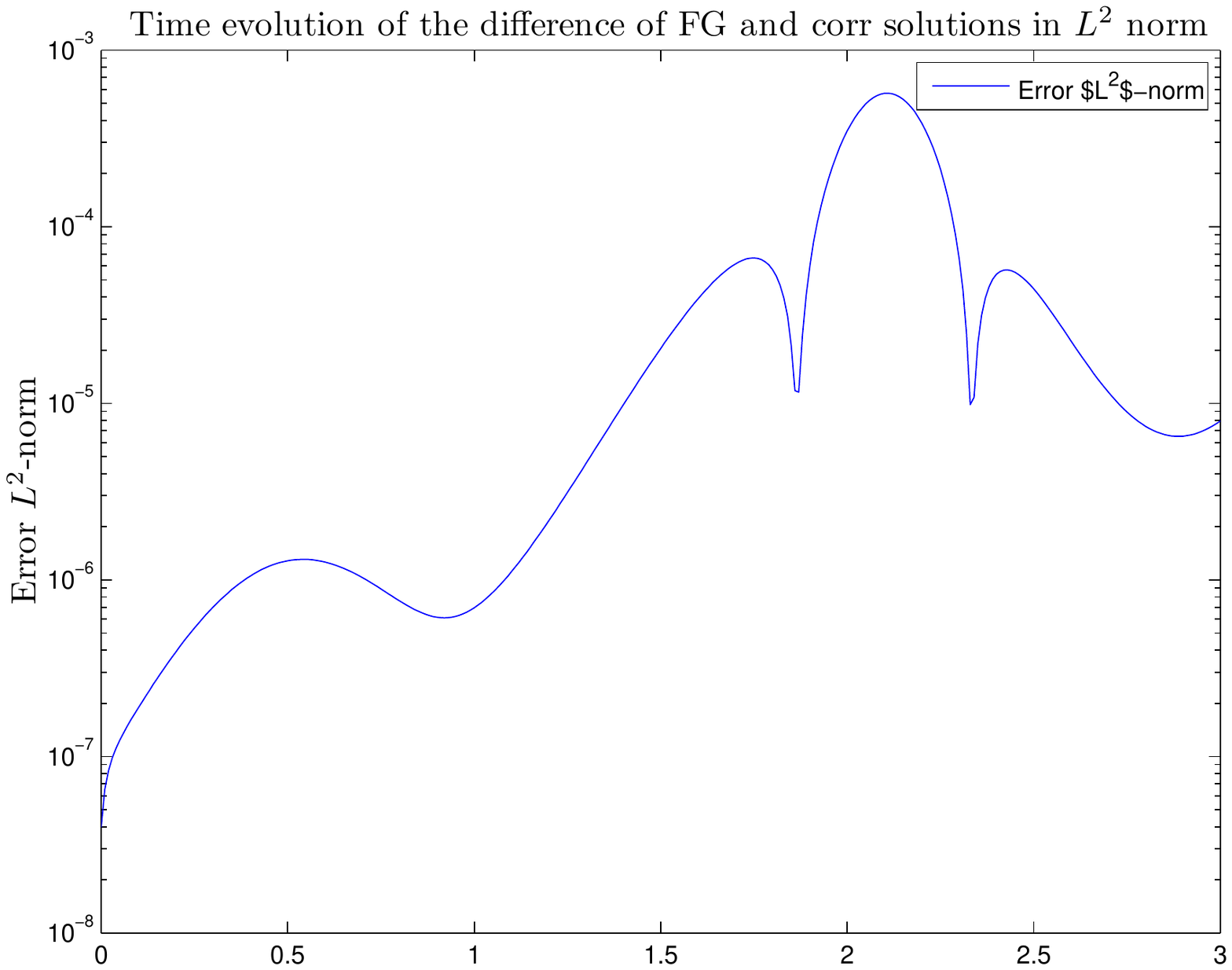}
\end{center}
\caption{Allen-Cahn - simulation of an exact solution
$\epsilon=0.3,\tau=1.8$, $\Delta t=10^{-2}$.}
\label{AC_Exact2}
\end{figure}
\clearpage
\section{Concluding Remarks}
The two-grid  method in Finite Elements for reaction-diffusion equations we have presented
here allows to produce fast and stable iterations, we gave also numerical evidences that our schemes capture important properties
such as energy diminishing, which is fundamental when considering gradient flow models . This can be extended to a larger number of nested spaces,
such as in \cite{DuboisJauberteauTemam1,DuboisJauberteauTemam2,HeLiuSon}, hoping to save much more computing time.\\
The separation of the scale allows to damp mainly the high mode components of the solution, this procedure can be interpreted as high-mode smoother.\\
An important advantage of the bi-grid framework is that it is not needed to work with embedded spaces,
particularly it is not necessary to compute a hierarchical-like basis, the filtering is automatically provided by the prolongation step. We have considered here first reaction-diffusion problems: it is a first study to be done and to be validated before applying the method to more complex systems.
%
\section*{Acknowledgment}
\noindent This project has been founded with support from the
National Council for Scientific Research in Lebanon, the Lebanese
University and C\`EDRE project "PHAFASA".

\end{document}